\documentclass[12pt]{amsart}
\usepackage{graphicx,psfrag,epsfig}
\usepackage{amsmath}
\usepackage{amsthm}
\usepackage[top=2.5cm, bottom=1.8cm]{geometry}
\usepackage{amsfonts}
\usepackage{amssymb}

\usepackage{color}

\newtheorem{lemma}{Lemma}[section]
\newtheorem{theorem}[lemma]{Theorem}
\newtheorem{corollary}[lemma]{Corollary}
\newtheorem{proposition}[lemma]{Proposition}

\theoremstyle{definition}
\newtheorem*{definition}{Definition}

\newtheorem{remark}[lemma]{Remark}

\newtheorem{example}[lemma]{Example}

\numberwithin{equation}{section}

\def\de{\overset{\mathrm{def}}{=}}

\def\mi\mathfrak{i}

\def\one{{\mathbf{1}}}

\def\DS{\displaystyle}

\def\diam{{\rm diam}}

\def\eps{{\varepsilon}}
\def\teps{{\tilde\varepsilon}}

\def\Card{{\rm Card}}
\def\Const{{\rm Const}}

\def\Prob{{\mathbb{P}}}

\def\EXP{{\mathbb{E}}}

\def\bbF{\mathbb{F}}

\def\bbL{\mathbb{L}}

\def\bbS{\mathbb{S}}

\def\naturals{\mathbb{N}}

\def\Tor{\mathbb{T}}

\def\reals{\mathbb{R}}

\def\integers{\mathbb{Z}}

\def\bE{\mathbf{E}}

\def\bP{\mathbf{P}}

\def\bd{\mathbf{d}}

\def\bv{\mathbf{v}}

\def\bx{\mathbf{x}}
\def\by{\mathbf{y}}

\def\brho{\boldsymbol{\rho}}

\def\bdeta{\boldsymbol{\eta}}

\def\bupsilon{{\boldsymbol{\upsilon}}}

\def\bDelta{{\boldsymbol{\Delta}}}

\def\brA{{\bar A}}

\def\brC{{\bar C}}

\def\brrC{{\bar\brC}}

\def\brK{{\bar K}}

\def\brN{{\bar N}}

\def\brP{{\bar P}}

\def\brR{{\bar R}}

\def\brX{{\bar X}}

\def\brY{{\bar Y}}

\def\brc{{\bar c}}

\def\brk{{\bar k}}

\def\brl{{\bar l}}

\def\brupsilon{{\bar\upsilon}}

\def\brv{{\bar v}}

\def\brx{{\bar x}}

\def\breps{{\bar\eps}}

\def\brlambda{{\bar \lambda}}

\def\brxi{{\bar \xi}}

\def\brrho{{\bar\rho}}

\def\brtheta{{\bar \theta}}

\def\brzeta{{\bar \zeta}}

\def\cA{\mathcal{A}}

\def\cB{\mathcal{B}}

\def\cI{\mathcal{I}}

\def\cF{\mathcal{F}}

\def\cG{\mathcal{G}}

\def\cH{\mathcal{H}}

\def\cM{\mathcal{M}}

\def\cN{\mathcal{N}}

\def\cO{\mathcal{O}}

\def\cP{\mathcal{P}}

\def\cQ{\mathcal{Q}}

\def\cR{\mathcal{R}}

\def\cS{\mathcal{S}}

\def\cT{\mathcal{T}}

\def\cV{\mathcal{V}}

\def\cW{\mathcal{W}}

\def\cX{\mathcal{X}}

\def\cY{\mathcal{Y}}

\def\fP{\mathfrak{P}}

\def\fQ{\mathfrak{Q}}

\def\fX{\mathfrak{X}}

\def\fh{\mathfrak{h}}

\def\fl{\mathfrak{l}}

\def\fm{\mathfrak{m}}

\def\fn{\mathfrak{n}}

\def\fp{\mathfrak{p}}

\def\fq{\mathfrak{q}}

\def\fs{\mathfrak{s}}

\def\ft{\mathfrak{t}}

\def\fz{\mathfrak{z}}

\def\hG{{\hat G}}

\def\hW{{\hat W}}

\def\hk{{\hat k}}

\def\hp{{\hat p}}

\def\heps{{\hat{\eps}}}

\def\hPhi{{\hat\Phi}}

\def\tB{{\tilde B}}

\def\tC{{\tilde C}}

\def\tD{{\tilde D}}

\def\tG{{\tilde G}}

\def\tN{{\tilde N}}

\def\tS{{\tilde S}}

\def\tX{{\tilde X}}

\def\tW{{\tilde W}}

\def\tY{{\tilde Y}}

\def\tc{{\tilde c}}

\def\tk{{\tilde k}}

\def\tfm{{\tilde{\fm}}}

\def\tp{{\tilde p}}

\def\tq{{\widetilde{q}}}

\def\tfq{{\widetilde{\fq}}}

\def\tu{{\tilde u}}

\def\ty{{\tilde y}}

\def\tbeta{{\tilde\beta}}

\def\tgamma{{\tilde\gamma}}

\def\tdelta{{\tilde\delta}}
\def\tDelta{{\tilde{\mathbf{\Delta}}}}

\def\teps{{\tilde\varepsilon}}

\def\tlambda{{\tilde\lambda}}

\def\tpi{{\tilde\pi}}

\def\ttau{{\tilde\tau}}

\def\trho{{\tilde\rho}}

\def\tzeta{{\tilde\zeta}}

\def\txi{{\tilde\xi}}

\def\tcB{{\tilde{\mathcal{B}}}}

\makeatother

\def\beq{\begin{equation}}
\def\eeq{\end{equation}}

\begin{document}
\title[Constructive approach to limit theorems for random walks]
{Constructive approach to limit theorems for recurrent diffusive random walks
on a strip.}

\author{D. Dolgopyat and I. Goldsheid}
\address{Dmitry Dolgopyat: Department of
Mathematics and Institute of Physical Science and Technology,
University  of Maryland, College Park, MD, 20742, USA}
\address{Ilya Goldsheid: School of Mathematical Sciences, Queen
Mary University of London, London
E1 4NS, Great Britain}

\keywords{random walks on a strip, Green function,  RWRE, CLT,
Local Limit Theorem, environment viewed from the particle}
\subjclass[2010]{Primary: 60K37; Secondary: 60F15}

\maketitle

\begin{abstract}
We consider recurrent diffusive random walks on a strip. We present constructive conditions  on Green functions
of finite sub-domains which imply a Central Limit Theorem with polynomial error bound, a Local Limit Theorem, and
mixing of environment viewed by the particle process. Our conditions can be verified for a wide class of environments
including independent environments, quasiperiodic environments, and environments which are asymptotically constant
at infinity. The conditions presented deal with a fixed environment, in particular, no stationarity conditions are imposed.
\end{abstract}

\section{Introduction.}
The one-dimensional random walk in random environment (RWRE) is a classical model in probability which was first considered in
\cite{So} and \cite{KKS} in 1975.
Remarkably, the behavior of the RWRE turns out to be quite different from that of the
simple random walk. Perhaps the most famous example of that is the theorem of Sinai (\cite{S1}) which states that
for the nearest neighbor random walks in the i.i.d. environment in the recurrent case
the walker typically is in a $\cO(\ln^2 n)$ neighborhood of
the origin after $n$ steps.

For walks on $\mathbb{Z}$ with bounded jumps, it was shown that in the recurrent case the Sinai behavior and the classical CLT
are the only possible scenarios for two important classes of environments. Namely, \cite{BG2} proves this for independent environments and
\cite{Br2} considers quasiperiodic Diophantine environments and proves that
the CLT holds with probability one.
Recently this result from \cite{Br2}  was extended in \cite{DG4} to RWRE on a strip, a natural generalization
of a random walk on $\integers$ with bounded jumps which was introduced in \cite{BG1}.
In fact, it is shown in \cite{DG4} that in both the  i.i.d and the quasiperiodic
Diophantine environments the
CLT holds in the recurrent case if and only if the potential is bounded (the precise definition of the
potential is given in Section \ref{SSBP}, see equation \eqref{defPhi}).
This is why it is natural and important to study recurrent RWs in a bounded potential.
The recurrent RWs in bounded potential are the main object studied in this paper.
However, in contrast to \cite{DG4} we deal with a fixed environment.
We develop a constructive approach which relates directly
the rate of convergence of ergodic averages for some specific observables to the CLT.
For a typical realization of a random environment our results recover the previously known
results and, moreover, we obtain new information also for RWRE. Namely, in the quasiperiodic Diophantine
case, the CLT is proven in \cite{DG4} only for a set of environments of full measure, while our present
methods imply that the CLT holds for all such environments without exception.

Our approach has several additional benefits.
It allows us to

\begin{itemize}
\item obtain explicit rate of convergence in the CLT;
\item establish the almost sure mixing estimates for environment seen by the particle
thus extending the results of \cite{L} and \cite{S2};
\item prove local limit theorems for several classes of environments;
\item apply our method to non stationary environments.
\end{itemize}

Let us describe the main novel techniques of the present paper which are crucial for our approach.
The first one is the asymptotic formula for the Green function in a large finite domain
obtained in Section \ref{ScGreen}.
The derivation of this formula relies on an entirely new approach to the analysis of the martingale
and the invariant measure equations which was recently discovered in \cite{DG5}. This approach is further
developed in this work and leads to new algebraic properties of the solutions to these equations.
The second key ingredient is the weak law of large numbers for the environment
viewed by the particle process. Our proof of the law of large numbers relies on the
Green function estimates and because of that is more transparent and applicable to a much wider class
of observables (the so called self-averaging observables) than the traditional approach
based on  the ergodic theorem (see e.g. \cite{BS, BS1}).

The layout of the paper is as follows. In Section \ref{ScModel} we define the model,
introduce some notation, and provide the necessary background for RW on a strip.
In particular we introduce RW in a bounded potential studied in \cite{DG4}.
In Section \ref{ScApplicatons} we illustrate the main results of our paper by applying them to several important classes of environments.
The precise formulation of the main results in the general case are given later since they are of a more technical nature.
Section \ref{ScGreen} contains bounds for the Green function of the walks with bounded potential.
In Section \ref{ScEnvLLN} these bounds are used to give a constructive proof of ergodicity for the environment viewed by
the particle process, giving a rate of convergence of time averages seen by the walker to the space averages.
In particular, this allows us to control the drift and the variance of the increments of the walker on a mesoscopic scale.
This allows us, in Section~\ref{ScCLT}, to obtain the Central Limit Theorem by the martingale method
and gives an estimate on the rate of convergence for several important classes of environments.
In Section \ref{ScSkew} we consider environments which have different asymptotic behaviors at $+\infty$
and $-\infty$ and show how the arguments of the previous section could be modified to obtain convergence
to the skew Brownian Motion type processes.
 In Section \ref{ScSemiLoc} we use a bootstrap
argument to show that the distribution of the walker's position is the same as for the Brownian Motion on a scale
which is slightly larger than $\cO(1).$ In Section~\ref{ScEnvMix} a local limit theorem for hitting times
is used to obtain mixing of the environment seen by the particle.  In Section~\ref{ScLLT} the results of Sections
\ref{ScSemiLoc} and \ref{ScEnvMix} are combined to obtain the local central limit theorem for the walker's position.

\section{Definition of the model and some preparatory results.}
\label{ScModel}

\subsection{Conventions and notation.}\label{notations}

The following notations and definitions are used throughout the paper.

All vectors and matrices below will be $m$-dimensional where $m$ is the width of the strip.
The dot product of vectors $x$ and $y$ will be denoted by $xy.$

$\mathbf{1}$ is a column vector whose components are all equal to $1$.

$e_i$ is the vector whose $i$-th component is 1 and all other components are 0.

For a vector ${x}=(x_{i})$ and a matrix $A=(a(i,j))$ we set
\[
\left\|  {x}\right\|  \overset{\mathrm{def}}{=}\max_{i}|x_{i}|\text{ which implies } \left\|
A\right\| =\sup_{\left\|  {x}\right\|  =1}\left\|
A{x}\right\|=\max_{i}\sum_{j}|a(i,j)|.
\]
We say that $A$ is strictly positive (and we write $A>0),$ if
all its matrix elements satisfy $a(i,j)>0$. $A$ is called non-negative (and we
write $A\geq0)$, if all $a(i,j)$ are non negative. A similar convention
applies to vectors. Note that if $A$ is a non-negative matrix then
$\left\|  A\right\| = \left\|
A{\one}\right\|$

$\bbS$ denotes the strip, $\bbS=\integers\times \{1, \ldots, m\}.$
Given a function $h:\mathbb{S}\mapsto\mathbb{R}$, we can define a sequence of $\mathbb{R}^m$-vectors
$h_n$ with components $h(n,i)$.
Vice versa, given a sequence of vectors $h_n$, $-\infty<n<\infty$, we define a function $h:\mathbb{S}\mapsto\mathbb{R}$ by
setting $h(n,i)=h_n(i)$, where $h_n(i)$ is the $i^{\textrm{th}}$ component of $h_n$.

\subsection{The Model.}\label{SubScModel}
We recall the definition of the RW on a strip from \cite{BG1}.
Let $\bbL_n=\{(n,i):\,1\le i\le m\}$ be layer $n$ of the strip, $\bbL_n\subset\mathbb{S}$. In our model, the walker is
allowed to jump from a point $(n,i)\in \bbL_n$ only to points in $\bbL_{n-1}$, or $ \bbL_n$, or
$\bbL_{n+1}$.
To define the corresponding transition kernel consider a sequence $\omega$ of triples
$(P_{n},Q_{n},R_{n}),\ -\infty<n<\infty,$ of $m\times m$ non-negative matrices such that for
all $n\in\mathbb{Z}$ the sum $P_{n}+Q_{n}+R_{n}$ is a stochastic matrix:
\begin{equation}\label{stoch}
(P_{n}+Q_{n}+R_{n})\mathbf{1}=\mathbf{1},
\end{equation}
We say that the sequence $\omega$ is the \textit{environment }on the strip $\mathbb{S}$.

The matrix elements of $P_{n}$ are denoted $P_{n}(i,j),$ $1\leq i,j\leq m,$ and
similar notations are used for $Q_{n}$ and $R_{n}.$ We now set
\begin{equation}
\mathfrak{P}(z,z_{1})\overset{\mathrm{def}}{=}\left\{
\begin{array}
[c]{ll}%
P_{n}(i,j) & \mathrm{if\quad}z=(n,i),\ z_{1}=(n+1,j),\\
R_{n}(i,j) & \mathrm{if\quad}z=(n,i),\ z_{1}=(n,j),\\
Q_{n}(i,j) & \mathrm{if\quad}z=(n,i),\ z_{1}=(n-1,j),\\
0 & \mathrm{otherwise.}%
\end{array}
\right.  \label{striptransition}%
\end{equation}

\begin{remark} The study of one-dimensional RW on $\mathbb{Z}$ with jumps of length $\le m$
can be reduced to the study of the above model by mapping $n\in\mathbb{Z}$ to
$\left(\lfloor\frac{n}{m}\rfloor, n-m  \lfloor\frac{n}{m}\rfloor\right)\in\mathbb{S}$,
where $\lfloor\cdot\rfloor$ denotes the integer part.
We refer the reader to \cite{BG1} for the formulas for transition matrices in that case and to
\cite{BG2} for more comments concerning this relationship.
\end{remark}

For a fixed $\omega$ we define a random walk $\xi(t)=(X(t),Y(t)),$ $t\ge 0$, on $\mathbb{S}$
in the usual way: for any starting point $z=(n,i)\in\mathbb{S}$ and fixed $\omega$
the law $\Prob_{\omega,z}$ for the Markov chain $\xi(\cdot)$ is given by
\begin{equation}
\Prob_{\omega,z}\left(  \xi(1)=z_{1},\ldots,\xi(t)=z_{t}\right)  \overset
{\mathrm{def}}{=}\mathfrak{P}_{\omega}(z,z_{1})\mathfrak{P}_{\omega}(z_{1},
z_{2})\cdots\mathfrak{P}_{\omega}(z_{t-1},z_{t}). \label{StripRWRE}%
\end{equation}
Let $\Xi_{z}$ be the set of trajectories $\xi(\cdot)$ starting at $z$. The just defined
$\mathbb{P}_{\omega,z}$ is a probability measure on $\Xi_{z}$; we denote by $\EXP_{\omega,z}$
the expectation with respect to this measure.
\begin{remark}
Fix $\mathfrak{i}\in \{1, \dots, m\}.$
With a slight abuse of notation, we shall often write $\mathbb{P}$ and $\mathbb{E}$ for
$\mathbb{P}_{\omega, (0,\mathfrak{i}) }$ and $\EXP_{\omega,(0, \mathfrak{i})}$.
We shall do that if the environment $\omega$ is
fixed. Since the strip has finite width, all the results proved in the paper will be uniform with respect
to $\mathfrak{i}.$

Also it will often be convenient to write $\xi_t$ for $\xi(t)$ and $X_t$ and $Y_t$ for $X(t)$ and $Y(t)$
respectively.
\end{remark}

Throughout the paper we suppose that the following ellipticity conditions are satisfied:
there is an $\breps>0$ and a positive integer number
$k_0<\infty$ such that for any $n\in\integers$ and any
$i,\,j\in[1,m]$
\begin{equation}\label{EqC2*}
||R_n^{k_0}||\le1-\breps,\ \ ((I-R_n)^{-1}P_n)(i,j)\ge\breps,
\text{ and } ((I-R_n)^{-1}Q_n)(i,j)\ge\breps.
\end{equation}
Note that $((I-R_n)^{-1}P_n)(i,j)$ (respectively $((I-R_n)^{-1}Q_n)(i,j)$) is the probability
that the walker starting from $(n,i)$ reaches $(n+1, j)$ (respectively $(n-1, j)$) at the first
exit from the layer $\bbL_n.$

\begin{remark}\label{Rem2.3}
Most of our results will be proved for environments which are not random but rather satisfy certain
properties (which will be listed in due time). We shall show that it is possible to apply these
results to certain important classes of random environments. More precisely, denote by
$(\Omega,\mathcal{F},\mathrm{P},T)$ the dynamical system where $\Omega$ is
the space of all sequences $\omega =((P_{n},Q_{n},R_{n}))_{n=-\infty}^\infty$ of
triples described above, $\mathcal{F}$ is the corresponding natural
$\sigma$-algebra, $\mathrm{P}$ denotes the probability measure on
$(\Omega,\mathcal{F})$, and $T$ is the shift operator on $\Omega$
defined by $T(P_{n},Q_{n},R_{n})=(P_{n+1},Q_{n+1},R_{n+1})$.
We shall always suppose that $T$ preserves measure $\mathrm{P}$ and is ergodic.
The expectation with respect to $\mathrm{P}$ will be denoted by $\mathrm{E}.$

To be able to apply the result obtained for deterministic environments
in the context of random environment, we shall check that the conditions we need
are satisfied by $\mathrm{P}$-almost all random environments.

\end{remark}

\begin{remark} Apart of the probability measures $\mathbb{P}$ and $\mathrm{P}$ defined above, we shall quite often
use measures which will be denoted by $\mathbf{P}$, with related expectations denoted $\mathbf{E}$, and which describe
'reference' probabilities and expectations related to, e. g., standard normal distribution, standard
Wiener processe, well known results concerning martingales, etc. Theorems \ref{ThCLTEx}, \ref{ThLLTEx},
Corollary \ref{CrExpVis}, Propositions \ref{Mart-BE}, \ref{PrConv} are examples
where this notation is used. In each such case, the precise meaning of $\mathbf{P}(\cdot)$ is obvious from the context.
\end{remark}

Denote by $\mathcal{J}$ the following set of triples of $m\times m$ matrices:
\[
\mathcal{J}\overset{\mathrm{def}}{=}\left\{  (P,Q,R)\,:\,P\ge0,\, Q\ge
0,\,R\ge0 \ \hbox{
and }\ (P+Q+R)\mathbf{1}=\mathbf{1}\right\}  .
\]

We shall use the following metric on $\Omega=\mathcal{J}^\mathbb{Z}.$
For
$\omega'=\{(P_n', Q_n', R_n')\},$ $\omega''=\{(P_n'', Q_n'', R_n'')\}$ set
\begin{equation}
\label{DefD}
\bd(\omega', \omega'')=\sum_{n\in\integers} \frac{\|P_n'-P_n''\|+\|Q_n'-Q_n''\|+\|R_n'-R_n''\|}{2^{|n|}}.
\end{equation}

Below, whenever we say that a function defined on $\Omega$ is continuous
we mean that it is continuous with respect to the topology induced by the metric $\bd(\cdot, \cdot)$.

\subsection{Matrices $\zeta_n$, $A_n$, $\alpha_n$ and some related quantities.}\label{ScMatrices}
We are now in a position to
recall the definitions of several objects most of which were first introduced and studied in \cite{BG1},
\cite{BG2} and which will play a crucial role in this work.

For a given $\omega\in\Omega$, define a sequence of $m\times m$ stochastic matrices $\zeta_n$ as follows.
For an integer $a$ let $\psi_a$ be a stochastic matrix. For $n>a$ define matrices $\psi_{n,a}$ recurrently as follows:
$\psi_{a,a}=\psi_a$ and
\begin{equation}
\label{psi}
\psi_{n,a}=(I-R_n-Q_n\psi_{n-1,a})^{-1}P_n, \quad n=a+1,\, a+2,\, \ldots.
\end{equation}
It is easy to show (see \cite{BG1}) that matrices $\psi_{n,a}$ are stochastic. Now for a fixed $n$ define
\begin{equation}
\label{zeta}
\zeta_n=\lim_{a\to -\infty}\psi_{n,a}.
\end{equation}
As shown in \cite[Theorem 1]{BG1} the limit \eqref{zeta} exists and is independent of the choice of
the sequence $\{\psi_a\}$.

Next, we define probability row-vectors $\sigma_n=\sigma_n(\omega)=(\sigma_n(\omega,1),\ldots,\sigma_n(\omega,m))$ which
are associated with the matrices $\zeta_n$. Let $\tilde{\sigma}_a$
be an arbitrary sequence of probability row-vectors (by which we mean that
$\tilde{\sigma}_a\ge 0$ and $\sum_{i=1}^m \tilde{\sigma}_a(i)=1$). Set
\begin{equation}\label{y}
\sigma_n\overset{\mathrm{def}}{=}\lim_{a\to-\infty}
\tilde{\sigma}_a\zeta_a\ldots\zeta_{n-1}.
\end{equation}
By the standard contraction property of the product of stochastic matrices, this limit exists
and does not depend on the choice of the sequence $\tilde{\sigma}_a$ (see \cite[Lemma 1]{G2}).
Vectors $\sigma_n$ could be equivalently defined as the unique sequence
of probability vectors satisfying the infinite system of equations
\begin{equation}\label{y1}
\sigma_n=\sigma_{n-1}\zeta_{n-1},\quad n\in\integers.
\end{equation}
Combining \eqref{y} with standard contracting properties of stochastic matrices
$\zeta$ we obtain for $k>n$ that
\begin{equation}\label{prodZeta}
\left\Vert\zeta_n \dots \zeta_{k-1}-(\sigma_k(1)\one,\dots,\sigma_k(m)\one )\right\Vert\le
 C\theta^{k-n},
\end{equation}
where $0\le\theta<1$ and $C$ depend only on the $\breps$ from \eqref{EqC2*}.

Define
\begin{equation}
\label{AAlpha}
\alpha_n=Q_{n+1}(I-R_n-Q_n \zeta_{n-1})^{-1}, \quad
A_n=(I-R_n-Q_n \zeta_{n-1})^{-1} Q_n.
\end{equation}
Note that $\alpha_nP_n=Q_{n+1}\zeta_n$ and hence
\begin{equation}
\label{AAlpha1}
\alpha_n=Q_{n+1}(I-R_n-\alpha_{n-1} P_{n-1})^{-1}.
\end{equation}
Conditions \eqref{EqC2*} imply (see \cite[Remark 2.2]{DG4}) that matrices $A_n$
have the following properties:
\begin{equation}\label{EqPositive}
||A_n||\le  (m\bar{\eps})^{-1}\text{ and } A_n(i,j)\ge\bar{\eps}
\end{equation}
In turn, inequalities \eqref{EqPositive} imply the well known contracting property of the action
of matrices $A_n$. We shall make use of the following version of this property: the limit
\begin{equation}\label{v_n}
v_n =\lim_{a\rightarrow-\infty}
\frac{A_{n}A_{n-1}\dots A_{a+1}\tilde{v}_{a}}{\left\|  A_{n}A_{n-1}\dots A_{a+1}\tilde{v}_{a}\right\|}
\end{equation}
exists and does not depend on the choice of the sequence of vectors $\tilde{v}_{a}\ge0,
||\tilde{v}_{a}||=1$. Moreover, there is a $\theta$, $0<\theta<1$, such that
\begin{equation}\label{v_n1}
\left\|v_n-
\frac{A_{n}A_{n-1}\dots A_{a+1}\tilde{v}_{a}}{\left\|  A_{n}A_{n-1}\dots A_{a+1}\tilde{v}_{a}\right\|}\right\|
=O(\theta^{n-a}).
\end{equation}
For the sake of completeness, we prove \eqref{v_n} and \eqref{v_n1} in Appendix \ref{AppPM}.

Similarly, for any sequence of row-vectors $\tilde{l}_a\ge0$, $\|\tilde{l}_a\|=1$, define
\begin{equation}\label{l_n}
l_n =\lim_{a\rightarrow\infty}\frac{
\tilde{l}_a\alpha_{a-1}\dots\alpha_{n}}{\left\|\tilde{l}_a\alpha_{a-1}\dots\alpha_{n}\right\|}.
\end{equation}
Set
\begin{equation}\label{EqLambda}
\lambda_k= \|A_k v_{k-1}\|\text{ and } \tlambda_k= \|l_{k+1}\alpha_{k}\|.
\end{equation}
Then obviously
\begin{equation} \label{EqPropl_n}l_{k+1}\alpha_{k}=\tlambda_k l_{k}, \quad
A_k v_{k-1}=\lambda_k v_k
\end{equation}
and for any $n\ge k$ we have
\begin{equation}\label{prodAandAlpha}
\left\|  A_{n}A_{n-1}\dots A_{k}v_{k-1}\right\|=\lambda_n\ldots\lambda_k, \quad
\left\|l_{n+1}\alpha_{n}\alpha_{n-1}\dots\alpha_{k}\right\|=\tlambda_n\ldots\tlambda_k.
\end{equation}

\begin{remark}
It should be
emphasized that even though \cite{BG1, BG2} dealt with stationary ergodic environments,
the proofs provided in \cite{BG1}, \cite{BG2} of the existence of the limits \eqref{zeta} and \eqref{v_n}
are in fact working for \textit{all} (and not just $\mathrm{P}$ - almost all) sequences $\omega$ satisfying \eqref{EqC2*}.
\end{remark}

\begin{remark}\label{Rem1D} Note that $m=1$ corresponds to the random walks on $\integers$
with jumps to the nearest neighbours. In this case $p_n=P_{\omega}(\xi(t+1)=n+1|\xi(t)=n)$ and $q_n=1-p_n$.
The above formulae now become very simple,
namely $\psi_n=\zeta_n=1$, $v_n=l_n=1$, $A_n=\lambda_n=\frac{q_n}{p_n}$, $\alpha_n=\tlambda_n=\frac{q_{n+1}}{p_n}$.
\end{remark}

In the above considerations, matrices $P_{n}$ and $Q_{n}$ play asymmetric roles
and it turns out to be useful to `symmetrize' the situation.
Namely, let us introduce stochastic matrices
$\zeta_{n}^{-}$ as the unique sequence of stochastic matrices satisfying the
system of equations which is symmetric to \eqref{psi}, \eqref{zeta}
\begin{equation}\label{EqZeta-}
\zeta_{n}^{-}=(I-R_{n}-P_{n}\zeta_{n+1}^{-})^{-1}Q_{n},\ -\infty
<n<+\infty.
\end{equation}
Next we set
\begin{equation}\label{DefA-}
A_{n}^{-}\overset{\mathrm{def}}{=}(I-R_{n}-P_{n}\zeta_{n+1}^{-})^{-1}
P_{n}, \quad \alpha_n^-=P_{n-1}(I-R_{n}-P_{n}\zeta_{n+1}^{-})^{-1}.
\end{equation}
All other related objects are introduced similarly.

Matrices $\zeta_n^-$, $\alpha_n^-$, $A_n^-$, etc have properties which are similar to those of
matrices $\zeta_n$, $\alpha_n$, $A_n$ etc listed above. All these objects will be used below without
further explanations.

\subsection{Walks in bounded potential.}
\label{SSBP}
In the context of random walks in random environments,
the notion of potential was introduced in \cite{S1} in the case of the walks on $\mathbb{Z}$
with jumps to nearest neighbors. The following extension of this definition to the case of random walks on a strip was
given in \cite{BG2}.

\begin{definition}
A \textit{potential} is a function of $n$ (and $\omega$) defined by
\begin{equation}\label{defPhi}%
\mathcal{U}_{n}(\omega)\equiv\mathcal{U}_{n}\overset{\mathrm{def}}{=}\left\{
\begin{array}
[c]{ll}%
\log||A_{n}...A_{1}|| & \mathrm{if\ }n\geq1\\
0 & \mathrm{if\ }n=0\\
-\log||A_{0}...A_{n+1}|| & \mathrm{if\ }n\leq-1
\end{array}
\right.
\end{equation}
We say that a potential is bounded if there is a constant $C_P$ such that
\begin{equation}
\label{BPot}
\big| \mathcal{U}_{n}\big|\le
C_P\text{ for all }n.
\end{equation}
\end{definition}

Bounded potentials appear naturally in the study of the following two classes of environments.
First, it has been proved in \cite{BG2} that the recurrence of a random walk in an i.i.d. environment on a strip
is equivalent to exactly one of two options: either the potential is bounded or it converges, after the diffusive
rescaling, to the Wiener process.
In the second case the walk exhibits the Sinai behavior (\cite{BG2}). Next, in \cite{DG4} it was shown that for quasiperiodic
environments with Diophantine frequencies the potential is bounded if and only if the random walk is recurrent.

\subsection{One useful property of a bounded potential.}
Properties \eqref{BPot} and \eqref{EqPositive} imply that there is a constant
$\tC_P>0$ such that for any vector $x\in\mathbb{R}^{m}$, $x\ge0$, ($x\not=0$), and for any $n> k$
\begin{equation}\label{BPNorm}
e^{-\tC_P}\|x\|\one \leq A_{n} \dots A_{k+1}x\leq e^{\tC_P}\|x\|\one.
\end{equation}
We shall check this statement for the case $k\ge1$ (other cases are similar).

Note that for any $k$ and $x\geq0$ the second inequality in \eqref{EqPositive} implies  $(A_kx)(i)\geq \breps\|x\|$
for all $i, 1\le i\le m$, and so $A_kx\geq \breps\|x\|\one$.

By \eqref{BPot}, $\|A_{k-1}...A_1\|=\|A_{k-1}...A_1\one\|\geq e^{-C_P}$ which is equivalent to
saying that there is $e_i$ such that $A_{k-1}...A_1\one\geq e^{-C_P}e_i.$ But then
\[
A_kA_{k-1}...A_1\one\geq e^{-C_P}A_ke_i\geq e^{-C_P} \breps\one.
\]
So $\one\le \breps^{-1}e^{C_P} A_k...A_1\one$ and hence
\[
A_n...A_{k+1}x\le \|x\|A_n...A_{k+1}\one\le \breps^{-1}e^{C_P}\|x\|
A_n...A_{k+1}A_k...A_1\one\leq\breps^{-1}e^{2C_P}\|x\|\one
\]
proving the second inequality in \eqref{BPNorm}.

Next, by the definition of the norm (and since matrices are positive) we have that
\[
A_k...A_1\one\leq e^{C_P}\one\ \text{ and hence }\ \one\geq e^{-C_P}A_k...A_1\one.
\]
Since $A_n...A_{k+1}x\geq \breps\|x\|A_n... A_{k+1}\one$, we obtain
\[
A_n...A_{k+1}x\geq \breps\|x\|A_n... A_{k+1}\one\ge \breps e^{-C_P}\|x\|A_n... A_{k+1}A_k...A_1\one
\ge \breps e^{-2C_P}\|x\|\one
\]
which proves the first inequality in \eqref{BPNorm}.

From now on we always suppose that the potential is bounded and {\em we assume for the rest
of the paper that \eqref{BPNorm} is satisfied.}

In our previous work we have shown that walks in a bounded
potential satisfy the following properties.

(I) There exists a non-constant sequence of column
vectors $\fm_n\in \mathbb{R}^m$ (with components $\fm_{n}(i)$) and a constant $K$ such that
\begin{equation}
\label{BndInc}
|\fm_{n'}(i')-\fm_{n''}(i'')|\leq K\text{ if }|n'-n''|\leq 1,
\end{equation}
and for all $n$
\begin{equation}
\label{MartEq}
 \fm_n=P_n \fm_{n+1}+R_n \fm_n+ Q_n \fm_{n-1}.
\end{equation}
The construction of the sequence $\fm_n$ is presented in \cite[Section 7]{DG4}.
We recall the probabilistic meaning of \eqref{MartEq}. Let $\fm:\mathbb{S}\mapsto\mathbb{R}$ be a function
on a strip and $\fm_n\in \mathbb{R}^m$ be a sequence of column vectors with components $\fm_n(i)=\fm(n,i)$.
If $\xi(t)=(X_t,Y_t),\ t\ge0,$ is the RW defined in Subsection \ref{SubScModel} then the process
$M(t)\de \fm(\xi_t)\equiv \fm(X_t,Y_t)\equiv\fm_{X_t}(Y_t)$ \text{ is a martingale }
if and only if the vectors $\fm_n$ satisfy \eqref{MartEq}.

(II) There exists a positive bounded solution $\rho_n=(\rho_n(1),...,\rho_n(m)),\ -\infty<n<\infty,$ to the equation
\begin{equation}
\label{RhoEq}
\rho_n=\rho_{n-1}P_{n-1} +\rho_nR_n +\rho_{n+1}Q_{n+1}
\end{equation}
which also satisfies $\rho_n=\rho_{n+1} \alpha_n, $ $\rho_{n+1}=\rho_n \alpha_{n+1}^-.$

Equation \eqref{RhoEq} appears in several contexts. First, for a fixed environment, it describes the
invariant measure for the walker. Second, in the case when we deal with stationary environment
the solution to \eqref{RhoEq} provides invariant densities for the environment viewed from the particle process.
We refer the reader to \cite{DG5} for a comprehensive analysis of this equation on the strip.
The invariant measure equation for the stationary walks on $\integers$ with bounded jumps
was studied in \cite{Br1}.

In accordance with conventions of \S \ref{notations}
we will often  write $\fm(\xi_t)$ instead of $\fm_{X_t}(Y_t)$ and $\rho(\xi_t)$ instead of
$\rho(X_t,Y_t)=\rho_{X_t}(Y_t).$


\section{Application of results to some classes of environments.}\label{ScApplicatons}
In this section we first discuss examples of important classes of environments. We then state the results
which we obtained for these environments as corollaries of our main and more general (but also more technical)
theorems proved in this paper.
\subsection{Classes of environments.}
\label{ScEx}
\begin{example}
\label{ExQP}
{\bf Quasiperiodic systems.} Consider the environment given by
$$(P, Q, R)_n=(\cP, \cQ, \cR)(\omega+n\gamma),$$
where $\omega, \gamma \in \Tor^d$, $\Tor^d$ is a $d$-dimensional torus, and $\cP, \cQ, \cR:\Tor^d\to\reals$ are $C^\infty$ functions.
$\gamma$ is called the rotation vector.
\end{example}

RWs in quasiperiodic environments received less attention than the walks discussed in the two other examples below,
the main references relevant for our work being \cite{Al, Br2, S2, KS1}. However, its continuous space
analogue, the quasiperiodic diffusion, is a classical object in the PDE literature, see \cite{JKO, KLO} and references therein.

For quasiperiodic environment there exists a continuous function $\lambda: \Tor^d\to \reals$ such that
$\lambda_n(\omega)=\lambda(\omega+n\gamma)$ ($\lambda_n$ is defined in \eqref{EqLambda}).
We say that
$\gamma$ is {\em Diophantine} if there are constants $K, \tau$ such that for each $k\in\integers^d\setminus 0$
we have
\begin{equation}
\label{DiophTau}
d(\gamma k , 2\pi \integers)\geq \frac{K}{|k|^\tau},
\end{equation}
where $d$ denotes the distance on the line.
If $\gamma$ is Diophantine then $\lambda\in C^\infty(\Tor^d),$ see Appendix \ref{AppReg}. The recurrence condition
\cite{BG1} amounts to
\begin{equation}
\label{QPRec}
\int_{\Tor^d} \ln \lambda(\omega) d\omega=0,
\end{equation}
where $d\omega$ is normalized Lebesgue measure on the torus.
It is proven in \cite{DG4} that if $\gamma$ is Diophantine then for every triple
$(\cP, \cQ, \cR)$ the CLT holds for almost all $\omega.$
We note that the Diophantine assumption \eqref{DiophTau} is necessary, since
\cite{DFS} gives examples showing that the CLT need not hold if \eqref{DiophTau} fails.
In this paper we obtain additional information in the case when \eqref{DiophTau} and \eqref{QPRec} hold.

\begin{example}
\label{ExInd}
{\bf Independent environments.} Here we suppose that $(P, Q, R)_n$ for different $n$
are independent and identically distributed.
\end{example}

The study of RWRE on $\mathbb{Z}$ goes back to \cite{KKS, S1, So}.
We refer the reader to \cite{Z} for a good overview of this subject. The papers most relevant to the present work are also described below after the formulations of Theorems \ref{ThCLTEx}, \ref{ThLLTEx}, \ref{ThEnvMixEx}. The study of the walks on the strip was
initiated in \cite{BG1}, the main references for limit theorems in this setting are
\cite{BG2, G2, DG2, DG4}.

In particular, for independent environments it was shown in \cite{BG2} that in the recurrent case the Sinai behavior is observed unless
$(P, Q, R)_n$ belong to a proper algebraic subvariety  in the space of transition matrices. The behavior of the walker
on this subvariety was investigated in \cite{DG4} where it was proven that the solutions to
\eqref{MartEq} and \eqref{RhoEq} with properties (I) and (II) exist.

\begin{example}
\label{ExPert}
{\bf Small perturbations of the simple random walk on $\integers$.}
 Consider a random walk on $\integers$ with $p_n=\frac{1}{2}-a_n,$ $q_n=\frac{1}{2}+a_n$ where
$a_n$ satisfy
$$ |a_n|\leq \frac{K}{|n|^\kappa+1}, \text{ where }\kappa>1. $$
\end{example}
The condition $\kappa>1$ is sufficient for recurrence (see e.g.  \cite{Du}).
However, for our results to apply we need one more condition, namely
\begin{equation}
\label{NNUnbiased}
\bupsilon=1 \quad \text{where}\quad \bupsilon=\prod_{n\in\integers} \left(\frac{p_n}{q_n}\right).
\end{equation}
Condition \eqref{NNUnbiased} appears to be restrictive. However, we will show in Corollary \ref{CrSkew}
that it is in fact necessary for the CLT to hold.

The study of environments where the limit $\displaystyle \lim_{n\to\pm \infty}p_n$ exists has a long history.
The limit theorems for
such walks go back to \cite{Lam, St}. The setting which perhaps is the closest to ours can be found in \cite{MW3} where
the Central Limit Theorem is obtained in the {\em transient} case.
Small perturbations of RWRE were studied
in \cite{GPS, MW1, MW2}.
We refer the reader to \cite{MPW} and references therein for a review of more recent developments.
In the present paper we show that such walks fit into the more general framework that we consider.

\subsection{The results.}
Next, we describe applications of the general theory developed in this paper to the
classes of environments described above. We assume throughout this section that the ellipticity
condition \eqref{EqC2*} holds and that the
walk is recurrent. In addition, we assume \eqref{BPNorm} (this assumption is only non-trivial
in Example \ref{ExInd} while in Examples \ref{ExQP} and \ref{ExPert} it follows from
recurrence and ellipticity).

We would like to emphasize that our results by no means are limited to Examples
\ref{ExQP}, \ref{ExInd}, and \ref{ExPert}. In fact, Theorems \ref{ThFCLT},
\ref{ThSkewEx}, \ref{ThCLTEx}, \ref{ThLLTEx}, and \ref{ThEnvMixEx} below will be obtained as corollaries
of more general results, namely,
Theorems \ref{ThCLTWalk}, \ref{ThSkewGen}, \ref{ThCLTRate}, \ref{ThEnvMix}, and \ref{ThLLT} respectively.
The statements of these general theorems are more technical and will be given in a due course, after we
introduce the necessary background.

Theorems \ref{ThFCLT}, \ref{ThCLTEx}, \ref{ThSkewEx}, \ref{ThLLTEx}, and \ref{ThEnvMixEx}
below are valid for all environments in Examples \ref{ExQP} and \ref{ExPert}
and for almost all environments in Example \ref{ExInd}. However in that last case we provide explicit conditions on
environment (see equations \eqref{RCMart}, \eqref{RC-Occ-V}, \eqref{RC-Occ-Q})
which guarantee the validity of these theorems.

Let $\cN$ be the standard normal random variable and $\Phi$
be the cumulative distribution function of $\cN.$

\begin{theorem}
\label{ThFCLT}
{\bf (Functional CLT)}
There is a constant $D>0$ such that the process
$W_N(t)=\frac{X(tN)}{\sqrt{N}}$ converges in law as $N\to\infty$ to $\cW(t)$-the Brownian Motion
with zero mean and variance $Dt.$
\end{theorem}

In fact, we can obtain the functional CLT also for perturbations of our environments which decay at infinity
sufficiently fast. Namely, consider a perturbation $\bar\fP$ of $\fP$\footnote{In the setting of Example \ref{ExPert}
this means that we allow the environments which do not satisfy \eqref{NNUnbiased}.
In fact, it follows from the explicit expression for $\fp$ in terms of $\bupsilon$ (see equation \eqref{UpsilonGen})
that in
Example \ref{ExPert} $\fp=\frac{1}{2}$ iff $\bupsilon=1.$}
such that
$$ \left|\bar\fP(z, z')-\fP(z, z')\right|\leq \frac{C}{|n|^\kappa+1}
\text{ where } z=(n,j) \text{ and }\kappa>1.$$
Let $\brxi(t)=(\brX(t), \brY(t))$ denote the walk in the perturbed environment.

\begin{theorem}
\label{ThSkewEx}
{\bf (Functional CLT for the perturbed walk)}
There exist constants $\fp$ and $D>0$ such that the process
$\frac{\brX(tN)}{\sqrt{N}}$ converges in law as $N\to\infty$ to the skew Brownian Motion
with zero mean, variance $Dt,$ and skewness parameter $\fp.$
\end{theorem}
The definition and basic properties of the skew Brownian Motion will be discussed in Section \ref{ScSkew}.
Here we just mention that one way to construct the skew Brownian Motion with skewness parameter $\fp$ is to take the scaling limit
for the random walk which is symmetric everywhere except the origin, and which moves to the right from
the origin with probability $\fp$ and to the left with probability $1-\fp$ (see \cite{HS}).

\begin{theorem}
\label{ThCLTEx}
{\bf (Effective CLT)} There are constants $D, \upsilon$ such that for each $\eps$ there is a constant $C_\eps$ such that
$$\sup_x \left|\mathbb{P}\left(\frac{X(N)}{\sqrt{DN}} \leq x\right)-\Phi(x)\right|\leq C_\eps N^{-(\upsilon-\eps)} . $$
\end{theorem}

\noindent
The exponent $\upsilon$ is explicit. Namely,
$ \upsilon=\frac{1}{8}$ in Examples \ref{ExQP}, \ref{ExInd}, and
$\upsilon=\min\left(\frac{\kappa-1}{2}, \frac{1}{8}\right)$
in Example \ref{ExPert}.


For the next two theorems we assume for Examples \ref{ExQP} and \ref{ExInd} that the random walk is lazy
in the sense that
\begin{equation}
\label{CanStay}
R_n(i,i)\geq\breps>0.
\end{equation}
\begin{remark}
Assumption \eqref{CanStay} is made for convenience only in order to simplify the statements.
Indeed assumption \eqref{EqC2*} implies that the walker can reach all points at the neighboring
layer by the time it changes layers. Therefore if we define the stopping times $\tau(n)$ by the conditions
$\tau(0)=0, \quad \tau(n+1)=\min(\tau>\tau(n): X_\tau\neq X_{\tau(n)})$ then
the accelerated walk $\xi^*(n)=\xi(\tau(2n))$ satisfies  \eqref{CanStay}. However the natural objects associated
with $\xi^*$ (such as solutions to \eqref{RhoEq} etc) have a more complicated form than for $\xi$ so
we prefer to impose \eqref{CanStay}.
\end{remark}

\begin{theorem}
\label{ThLLTEx}
{\bf (Local Limit Theorem)}
(a) In Examples \ref{ExQP} and \ref{ExInd} there are constants $a, b$ such that uniformly for  $k_N/\sqrt{N}$ in a compact set
for each $y\in \{1,\dots, m\}$ we have
\begin{equation}\label{LLT-Rho}
\lim_{N\to\infty}  \frac{\mathbb{P}\left(\xi(N)=(k_N, y)\right)}{\bP\left(\sqrt{\frac{bN}{a}}\cN\in
\left[k_N-\frac{1}{2}, \, k_N+\frac{1}{2}\right]\right)\rho(k_N,y)}=\frac{1}{a}.
\end{equation}

(b) In Example \ref{ExPert}  uniformly for
$k_N/\sqrt{N}$ in a compact set if $k_N$ and $N$ have the same parity
then
\begin{equation}
\label{LLT-RhoPert}
\lim_{N\to\infty}  \frac{\mathbb{P}\left(\xi(N)=k_N\right)}{\mathbf{P}\left(\sqrt{N}\cN\in
\left[k_N-\frac{1}{2}, \, k_N+\frac{1}{2}\right]\right)\rho(k_N)}=2.
\end{equation}
\end{theorem}
\begin{remark}
Equation \eqref{RhoEq} defines $\rho$ up to a multiplicative constant. So to complete the statement
of Theorem \ref{ThLLTEx} one needs to explain how to normalize $\rho.$ This will be done in
Section \ref{ScGreen} (see equations \eqref{XMNorm} and \eqref{RhoMNorm}).

It will be shown in Section \ref{ScCLT} (see equation \eqref{RM-SRW}) that
with this choice of normalization, we have in Example \ref{ExPert} that
$\displaystyle \lim_{|k|\to\infty} \rho(k)=1.$ Thus if in Theorem \ref{ThLLTEx}(b) $|k_N|\to\infty$ then \eqref{LLT-RhoPert}
can be simplified to read
$$ \lim_{N\to\infty}  \frac{\mathbb{P}\left(\xi(N)=k_N\right)}{\bP\left(\sqrt{N}\cN\in
\left[k_N-\frac{1}{2}, \, k_N+\frac{1}{2}\right]\right)}=2. $$
That is in that case the Local Limit Theorem takes the same form as for the simple random walk.
\end{remark}
While the Central Limit Theorem was
studied for many classes of RWRE, the Local Limit Theorem is
less well understood. We note that there are two different classes of walks where the CLT
is known and so it
makes sense to study the LLT: ballistic walks are investigated in \cite{BCR, DG3, LS, S2} and
balanced walks in \cite{CD, S2, Stn}. In both cases the Local Limit Theorem takes the same form
\eqref{LLT-Rho},
but the meaning of $\rho$ is different: for ballistic walks $\rho_z$ is the expected number of visits to
the site $z$ while for recurrent walks it is proportional to the invariant measure of the walk restricted to a finite domain.
For this reason different methods are usually employed to study these two cases. In the present paper we adapt the method
used in \cite{DG3} to study the ballistic walks to the recurrent case (our approach is a modification of the method of
\cite{G2} and is related to extraction of a binomial component approach used in \cite{DMD}).
The universality of this method makes it promising in other problems where the Local Limit Theorem can be expected.

\smallskip
To formulate our last result we need one more definition.
In Examples \ref{ExQP} and \ref{ExInd}
a bounded function $h:\mathbb{S}\to \reals$ will be called {\em self-averaging}
if there is a constant $\fh$
(the average of $h$) and a sequence $\delta_N$
converging to $0$ as $N\to\infty$ such that for each $\eps, K$
for each $k$ with $|k|\leq K \sqrt{N}$
\begin{equation}
\label{UnifConv1}
\left|\frac{1}{2 \delta_N N^{1/4}}\sum_{l=k-\delta_N N^{1/4}}^{k+\delta_N N^{1/4}}\rho_l h_l
-\fh\right|\leq \eps
\end{equation}
where $h_l$ is a vector with components $h_l(j)=h(l,j),$
$\rho_l$ is the vector defined in \eqref{RhoEq}, whose components are denoted by $\rho_l(j),$ and
$$ \rho_l h_l=\sum_{j=1}^m \rho_l(j) h(l,j). $$
In Example \ref{ExPert} the
 walk is periodic with period 2, so
\eqref{UnifConv1} has to be replaced by
{\footnotesize
\begin{equation}
\label{UnifConv-Per}
\left|\frac{1}{2 \delta_N N^{1/4}}\sum_{l=-\delta_N N^{1/4}}^{\delta_N N^{1/4}}\rho_{k+2l} h_{k+2l}
-\fh\right|\leq \eps
\text{ and }
\left|\frac{1}{2 \delta_N N^{1/4}}\sum_{l=-\delta_N N^{1/4}}^{\delta_N N^{1/4}}\rho_{k+1+2l} h_{k+1+2l}
-\fh\right|\leq \eps.
\end{equation}}
The meaning of the notion of the self-averaging will be explained later (see Remark~\ref{RemExt}).

\begin{theorem}
\label{ThEnvMixEx}
{\bf (Mixing of environment viewed by the particle)}
If $h:\mathbb{S}\to\reals$ is self-averaging
then 
$$ \lim_{N\to\infty} \mathbb{E}(h(\xi(N)))=\frac{\fh}{a}. $$
where $a$ is the same as in \eqref{LLT-Rho}
\end{theorem}

\begin{remark}
The term {\em mixing} here refers to the fact that the expectation above
is asymptotically independent of $N.$
It follows from our proof that it is also independent of the starting point of the walk. Therefore Theorem \ref{ThEnvMixEx}
shows that the environment seen by the walker does not remember the remote past of the walker. Results
similar to Theorem~\ref{ThEnvMixEx} are sometimes called {\em renewal theorems} since mixing for certain
systems allows us to recover the classical renewal theorems.
We prefer the term {\em mixing} since it appears
to describe the phenomenon more precisely.
\end{remark}

The environment viewed by the particle process is a standard tool in studying the random
walk \cite{Koz, BS, BS1}.
For ballistic nearest neighbor random walks on $\integers$ in independent environments
the mixing of this process was obtained in \cite{K} in
the annealed setting. The quenched result  was  proven in \cite{L} for independent
walks under the additional assumption that  the fluctuations are diffusive (see \cite{DG3} for a simple
proof). \cite{KS1, S2}  prove mixing for quasiperiodic walks.
The results of \cite{K} have been extended to walks on
the strip in \cite{R}. In this paper we obtain quenched mixing in both independent and quasiperiodic
environments. In fact, the novel feature of our results is that they are
applicable to the environments satisfying
explicit estimates, so no stationarity is required in our approach.

\section{The Green Function.}
\label{ScGreen}

The main result of this section is the asymptotic expansion of the Green function for the exit from a large interval
(see Lemma~\ref{LmGF}). This asymptotic expansion plays a major role in the proofs
of our main results: it allows us to compute limits of ratios of various additive functionals of our random walk using
moderate deviation estimates from Appendix \ref{AppLargeMod}.

We begin with a preliminary fact, establishing a relation between two key quantities $\fm_n$ and $\rho_n$
which appear in the expansion of the Green function

\begin{lemma}
\label{LmCurrent}
If $\fm_n$ satisfies \eqref{MartEq} and $\rho_n$ satisfies \eqref{RhoEq} then
there exist a constant $c$ such that for all $n$
$$ \rho_{n+1} Q_{n+1} (\fm_n-\zeta_n \fm_{n+1})=c, $$
$$ \rho_n P_n (\fm_{n+1}-\zeta_{n+1}^- \fm_n)=-c. $$
\end{lemma}

This lemma complements \cite[Lemmas 4.5 and 4 .6]{DG5} where other relations between
$\rho_n$ and $\fm_n$ are described.

\begin{proof}
Let
\begin{equation}\label{DefU}
u_n=\fm_n-\zeta_n \fm_{n+1}.
\end{equation}
Then
\begin{equation}
\label{URec}
u_n=A_n u_{n-1}.
\end{equation}
Indeed \eqref{MartEq} can be rewritten as
\begin{equation}
\label{MartEq2}
(I-R_n) \fm_n=P_n \fm_{n+1}+ Q_n \fm_{n-1}
\end{equation}
Since $\zeta_n=(I-R_n-Q_n \zeta_{n-1})^{-1} P_n $
we have
$$ P_n=(I-R_n-Q_n \zeta_{n-1}) \zeta_n. $$
Plugging this into \eqref{MartEq2} we get
$$ (I-R_n) \fm_n= (I-R_n-Q_n \zeta_{n-1}) \zeta_n \fm_{n+1}+ Q_n \fm_{n-1}. $$
Subtracting $Q_n \zeta_{n-1} \fm_n$ from both sides we get
$$ (I-R_n-Q_n \zeta_{n-1}) \fm_n= (I-R_n-Q_n \zeta_{n-1}) \zeta_n \fm_{n+1}+
Q_n (\fm_{n-1}-\zeta_{n-1}\fm_n) $$
Multiplying both sides by $(I-R_n-Q_n \zeta_{n-1})^{-1}$ and remembering that
$$A_n=(I-R_n-Q_n \zeta_{n-1})^{-1} Q_n$$ we obtain \eqref{URec}.

Observe that \eqref{AAlpha} implies that
$Q_{n+1} A_n=\alpha_n Q_n.$ Hence \eqref{URec} gives
$$ \rho_{n+1} Q_{n+1} u_n=\rho_{n+1} Q_{n+1} A_n u_{n-1}=
\rho_{n+1} \alpha_n Q_n u_{n-1}=\rho_n Q_n u_{n-1}$$
proving the first claim of the lemma. A similar computation shows that
$$ \rho_n P_n (\fm_{n+1}-\zeta_{n+1}^- \fm_n)=c^-. $$
It remains to relate $c$ to $c^-.$ To this end note that
\begin{align*} c&=\rho_{n+1} Q_{n+1} (\fm_n-\zeta_n \fm_{n+1})\\
&=\rho_{n+1} Q_{n+1} \fm_n-\rho_{n+1} Q_{n+1}(I-R_n-Q_n \zeta_{n-1})^{-1} P_n \fm_{n+1}\\
&=\rho_{n+1} Q_{n+1} \fm_n-\rho_{n+1} \alpha_n P_n \fm_{n+1}= \rho_{n+1} Q_{n+1} \fm_n -\rho_n  P_n \fm_{n+1}.
\end{align*}
Likewise
$$ c^-=\rho_n  P_n \fm_{n+1}-\rho_{n+1} Q_{n+1} \fm_n=-c$$
finishing the proof.
\end{proof}

Next, \eqref{BPNorm} and property \eqref{EqPositive} imply the following Lemma.
\begin{lemma} Suppose that \eqref{EqC2*} is satisfied and the potential \eqref{defPhi} is bounded.
Then:
\newline $(\mathrm{i})$ there is a bounded solution $u_n$ to $u_n=A_n u_{n-1}, $ $\infty<n< \infty$,
and such that all entries of all vectors $u_n$ have the same sign.
\newline $(\mathrm{b})$ if a sequence of vectors $\tu_n$ satisfies $\tu_n=A_n \tu_{n-1}, $ $\infty<n< \infty$, and
$\frac{1}{n}\ln \|\tu_n\|\to 0$ as $n\to -\infty$ then $\tu_n$ is bounded and proportional to $u_n$.
\end{lemma}
\begin{proof} It will be convenient to use the following notation: for $k\le n$
\[
\mathbf{A}_n^k=A_n...A_{k+1}, \text{ with the convention } \mathbf{A}_{n}^{n}=I, \ \ \mathbf{A}_{n}^{n-1}=A_n.
\]
We make use of $v_n$ from \eqref{v_n}. To construct $u_n$ for $n\ge 0$
set $u_0=v_0$ and define $u_n=\mathbf{A}_n^0u_{0}$ for $n\ge 1$.  Obviously, $u_n=A_n u_{n-1}$ if $n\ge 1$.

For $n\le -1$, set $u_n\de\lambda_0^{-1} ... \lambda_{n+1}^{-1}v_{n}$. Then, taking into account \eqref{EqPropl_n}, we have for $n\le0$:
\[
A_n u_{n-1}=\lambda_0^{-1} ... \lambda_{n+1}^{-1}\lambda_{n}^{-1}A_nv_{n-1}=
\lambda_0^{-1} ... \lambda_{n+1}^{-1}v_{n}= u_n.
\]
The vectors $u_n$ are strictly positive since $v_n>0$ for all $n$.
Finally, since $u_n= A_{n} \dots A_{k+1}u_k$ for any $n> k$,
the inequalities \eqref{BPNorm} imply that $e^{-\tC_P}\|u_k\|\leq \|u_n\|\leq e^{\tC_P}\|u_k\|$ and so
this solution is bounded. This completes the proof of (i).

Proof of (ii). We shall show that for any fixed $n$ there is a $c$ such that
$\tu_n=cu_n$.

For $k\le n$, present $\tu_{k}=\tu_{k}^+-\tu_{k}^-$, where $\tu_k^+\ge0$ and $\tu_k^-\ge0$ are,
respectively, the positive and the negative part of $\tu_{k}$. Then $\|\tu_k^+\|\le \|\tu_k\|$ and $\|\tu_k^-\|\le \|\tu_k\|$.
It follows from \eqref{BPNorm} that $\|\mathbf{A}_{n}^{k}\tu_{k}^+\|=\cO(\|\tu_{k}^+\|)=\cO(\|\tu_{k}\|)$. This, together with \eqref{v_n1}
implies $\mathbf{A}_{n}^{k}\tu_{k}^+=\|\mathbf{A}_{n}^{k}\tu_{k}^+\| (v_n+\cO(\theta^{n-k}))=
\|\mathbf{A}_{n}^{k}\tu_{k}^+\| v_n+ \cO(\|\tu_k\|\theta^{n-k})$. Finally,
\[
\mathbf{A}_{n}^{k}\tu_{k}^+=
\|\mathbf{A}_{n}^{k}\tu_{k}^+\| \|u_n\|^{-1}u_n+ \cO(\|\tu_k\|\theta^{k-n})
\]
since $v_n={u_{n}}/{\|u_n\|}$. Similarly, $\mathbf{A}_{n}^{k}\tu_{k}^-=
\|\mathbf{A}_{n}^{k}\tu_{k}^+\| \|u_n\|^{-1}u_n+ \cO(\|\tu_k\|\theta^{n-k})$.
But $\tu_{n}=\mathbf{A}_{n}^{k}\tu_{k}^+ - \mathbf{A}_{n}^{k}\tu_{k}^-$ we have
\[
\tu_{n}=(\|\mathbf{A}_{n}^{k}\tu_{k}^+\| -\|\mathbf{A}_{n}^{k}\tu_{k}^-\| )\|u_n\|^{-1}u_n+\cO(\|\tu_k\|\theta^{n-k}).
\]
The growth of $\|\tu_{k}^-\|$ is sub-exponential and therefore, sending $k\to-\infty$, we see that
$c_n\de\lim_{k \to-\infty}(\|\mathbf{A}_{n}^{k}\tu_{k}^+\| -\|\mathbf{A}_{n}^{k}\tu_{k}^-\| )\|u_n\|^{-1}$
exists and $\tu_{n}=c_nu_{n}$. But then $c_nu_n=c_{n-1}A_nu_{n-1}$ by the definition of $\tu_n$ which,
together with \eqref{URec}, implies that $(c_n-c_{n-1})u_n=0$. Hence $c_n=c_{n-1}=\mathrm{const}$.
\end{proof}

Returning to the main equation \eqref{MartEq} for $\fm_n$, we have now two possibilities.

(1) $u_n\equiv 0.$ In this case, \eqref{DefU} implies that for each $k\ge 1,$
$\fm_n=\zeta_n \zeta_{n+1}\dots \zeta_{n+k} \fm_{n+k}.$
Sending $k$ to infinity and using contracting properties of stochastic matrices  we see that
$\fm_n\equiv c \one $ for some constant $c.$

(2) $u_n$ is non zero, so its entries have the same sign. Then $c=\rho_{n+1} Q_{n+1} u_n \neq 0.$

We note that in the second case the martingale increases faster than some linear function.
Namely there are constants $C_1, C_2$ such that for any $n,k\in \integers$, $k\ge0$ and for any
 $i,j\in \{1\dots m\}$ we have
 \begin{equation}
 \label{LocGrowth}
\frac{k}{C_1}-C_2\leq \fm_{n+k}(j)-\fm_n(i)\leq C_1 k+C_2.
 \end{equation}
 Indeed, iterating \eqref{DefU} we obtain
  \begin{equation}
 \label{MultiStepMU}
 \fm_n=\sum_{r=0}^{k-1} u_{n,r}+ \zeta_n\dots \zeta_{n+k-1} \fm_{n+k}
  \end{equation}
 where $u_{n,r}=\zeta_n\dots \zeta_{n+r-1} u_{n+r}.$
Note that by \eqref{BPNorm} the components of $u_n$ are uniformly bounded from above
and bounded away from zero.
Since $\zeta$'s are stochastic, we conclude that
the components of $u_{n,r}$ are uniformly bounded from above and bounded away from zero.
Since $\fm$ has bounded increments we have that for each $j=1,\dots, m,$
$\fm_{n+k}=\fm_{n+k}(j)\one+\cO(1). $ Plugging this into \eqref{MultiStepMU} and
using that $\zeta_n\dots \zeta_{n+k-1}\one=\one$ we obtain \eqref{LocGrowth}.

In this paper we deal with the case where the martingale $\fm$ is non-trivial. Moreover,
for the rest of the paper (except for the Section \ref{ScSkew})
we assume that $\fm$ is asymptotically linear and that $\fm$ and
$\rho$ are normalized so that
\begin{equation}
\label{XMNorm}
\lim_{n\to \infty} \frac{\cM_n}{n} =m,
\end{equation}
where $\cM_n=\sum_{j=1}^m \fm_n(j)$
and
\begin{equation}
\label{RhoMNorm}
\rho_n P_n (\fm_{n+1}-\zeta_{n+1}^- \fm_n)=\frac{1}{2m}.
\end{equation}
Note that \eqref{XMNorm} and \eqref{BndInc} imply that
\begin{equation}
\label{XMNorm2}
\lim_{n\to\infty} \frac{\fm_n(y)}{n} =1
\end{equation}
uniformly in $y\in \{1,\dots, m\}.$

Asymptotically linear martingales exist in Examples \ref{ExQP}--\ref{ExPert}.
In fact, in Section \ref{SSCLTEx}
we will establish a stronger
result \eqref{RCMart}.


Consider the Green function
$$G_{a,b}((k, i); (n,j))=\mathbb{E}(\eta_{(n,j)}|\xi(0)=(k,i)),$$
where $\eta_{(n,j)}$ is the number of visits to $(n,j)$ by the walk starting at $(k,i)$ before it
hits the segment $[a, b].$

\begin{lemma} For $x,\ y\in (a,b)$
\label{LmGF}
$$G_{a,b}((k,i); (n,j))=\cG(\cM_k, \cM_n, \cM_a, \cM_b) \rho_n(j)+\cO(1) $$
where
$$\cG(x,y, a, b)=\frac{2(\min(x,y)-a)(b-\max(x,y))}{(b-a)}. $$
\end{lemma}

\begin{proof}
Let
\begin{equation}
\label{SumGreen}
\tG_{a,b}((k,i); n)=\sum_j G_{a,b}((k,i), (n, j)).
\end{equation}
A simple computation with Markov chains \cite[Appendix A]{DG2} shows that
$$ G_{a,b}((k, i), (n, j))=\tG_{a,b}((k, i), n) \frac{\rho_n(j)}{\rho_n \one }+\cO(1) $$
so it suffices to show that
\begin{equation}
\label{GFLayer}
 \tG_{a,b}((k,i); n)=\cG(\cM_k, \cM_n, \cM_a, \cM_j)  \rho_n\one +\cO(1).
\end{equation}
We consider first the case where $k=n.$ Let $d=\min(n-a, b-n).$
Denote $\bdeta_n=\eta_n \one.$
Let $\fp_{n, a, b}$ be the vector
$$\fp_{n, a, b}(i)=\mathbb{P}(\bdeta_n=1|\xi(0)=(n, i)). $$

We claim that
\begin{equation}
\label{NoReturn}
\fp_{n, a, b}=\cO(1/d).
\end{equation}
Without loss of generality we may assume that $d=b-n.$
By \eqref{LocGrowth} there exists $k$ such that for any $n, i, j$
$\fm_{n+k}(j)\geq \fm_n(i)+1.$
Consider our walk started from $(n+k,j).$ Let $\fs$ be the first time this walk reaches
either $\bbL_n$ or $\bbL_b.$
Applying the Optional Stopping Theorem to this stopping time gives
$$ \fm_{n+k}(j)=\sum_{i=1}^m \Prob(\xi_\fs=(n,i))\fm_n(i)+
\sum_{i=1}^m \Prob(\xi_\fs=(b,i))\fm_b(i).$$
Rewriting this identity as
\begin{equation}
\label{OST-Prob}
 \sum_{i=1}^m \Prob(\xi_\fs=(n,i))[\fm_{n+k}(j)-\fm_n(i)]=
\sum_{i=1}^m \Prob(\xi_\fs=(b,i))[\fm_b(i)-\fm_{n+k}(j)]
\end{equation}
By our choice of $k$ the LHS is at least
$1-\Prob(\xi_\fs\in\bbL_b)$ while by \eqref{LocGrowth}
the RHS is at most $2 C_1 d\; \Prob(\xi_\fs\in \bbL).$
Thus
$$ 1\geq 2C_1 d \Prob(\xi_\fs\in \bbL) \text{ or } \Prob(\xi_\fs\in\bbL_b)=\cO(1/d). $$
In other words, if the walker starts from the layer $n+k$ then the probability
that it would not visit  $\bbL_n$ before reaching $\bbL_b$ is $\cO(1/d).$
By the same argument,
if the walker starts from the layer $n-k$ then the probability
that it would not visit  $\bbL_n$ before reaching $\bbL_a$ is $\cO(1/d).$
Since the walker starting from $\bbL_n$ should visit $\bbL_{n-k}$ before
reaching $\bbL_a$ and it should visit $\bbL_{n+k}$ before reaching $\bbL_b,$
\eqref{NoReturn} follows.
\medskip

Let $\tS_n$ be the matrix with components $\tS_n(i,j)$ where $\tS_n(i,j)$ is the probability that the
walker starting from $(n,i)$ returns to $\bbL_n$ for the first time at $(n,j)$ given that it does not visit
$\bbL_a$ or $\bbL_b$ in between. Let $\tpi_n$ be the stationary distribution for $\tS_n.$
Also let $S_n$ be the matrix with components $S_n(i,j)$, where $S_n(i,j)$ is the probability that the
walker starting from $(n,i)$ returns to $\bbL_n$ for the first time at $(n,j).$ Thus
$$ S_n=Q_n \zeta_{n-1}+R_n+P_n \zeta^-_{n+1}. $$
Let $\pi_n$ be the stationary distribution of $S_n.$ Note that $\pi_n$ can be expressed in terms of
$\rho_n$ as $\pi_n(j)=\frac{\rho_n(j)}{\rho_n \one}.$ Since, by \eqref{NoReturn},
 $S_n$ and $\tS_n$ differ by conditioning on a set
of measure $1-\cO\left(\frac{1}{d}\right)$ we have
$\tS_n=S_n+\cO\left(\frac{1}{d}\right).$
Moreover we can write
$$ \tS_n=Q_n \tzeta_{n-1}+R_n+P_n \tzeta^-_{n+1} $$
where $\tzeta_{n-1}(i,j)$ and $\zeta^-_{n+1}(i,j)$ are the  probabilities that the walker
starting from $(n-1,i)$ (respectively $(n+1, i)$) returns to $\bbL_n$ for the first time at $(n,j)$ given that it does not visit $\bbL_a$ or $\bbL_b$ in between.
Then
\begin{equation}
\label{ZetasClose}
\tzeta_{n-1}=\zeta_{n-1}+\cO\left(\frac{1}{d}\right), \quad
\tzeta_{n+1}^-=\zeta_{n+1}^-+\cO\left(\frac{1}{d}\right).
\end{equation}

Due to exponential mixing of both $S_n$ and $\tS_n$
(which is guaranteed by condition \eqref{EqC2*}) we have that
$$ \tpi_n=\pi_n+\cO\left(\frac{1}{d}\right)=\frac{\rho_n}{\rho_n \one}+\cO\left(\frac{1}{d}\right). $$
We have
$$ \mathbb{P}(\bdeta_n=N+1|\bdeta_n>N, \xi(0)=(n, i))=e_i \tS_n^N \fp_{n, a, b} .$$
By the first step analysis
\begin{equation}
\label{FirstStep}
 \fp_{n, a, b}(i)=\sum_{j=1}^m Q_n(i,j) \fp_n^-(j)+\sum_{j=1}^m P_n(i,j) \fp_n^+(j)
\end{equation}
where $\fp^\pm_n(j)$ is the probability that the walker starting from $(n\pm 1, j)$ does not return
to $\bbL_n$ before visiting the boundary of the segment $a,$ $b.$
By the Optional Stopping Theorem and \eqref{BndInc}
\begin{equation}
\label{PPlMin}
 \fp_n^-=\frac{ \tzeta_{n-1}\fm_n-\fm_{n-1}}{\frac{1}{m} (\cM_n-\cM_a)}+\cO\left(\frac{1}{d^2}\right),
\quad
\fp_n^+=\frac{ \tzeta_{n+1}^-\fm_n - \fm_{n+1}}{\frac{1}{m}(\cM_b-\cM_n)}+\cO\left(\frac{1}{d^2}\right).
\end{equation}
From \eqref{FirstStep} and \eqref{PPlMin} we get using \eqref{ZetasClose} that
\begin{equation}
\label{P-Zetas}
 \fp_{n, a, b}(i)=m e_i \left(Q_n \frac{ \zeta_{n-1}\fm_n-\fm_{n-1}}{\cM_n-\cM_a}+
  P_n \frac{ \zeta_{n+1}^-\fm_n - \fm_{n+1}}{\cM_b-\cM_n}\right)+\cO\left(\frac{1}{d^2}\right).
  \end{equation}

Since $e_i \tS_n^N=\tpi_n(1+\cO(\theta^N))$ for some $\theta<1$,
it follows that
$$ e_i \tS_n^N \fp_{n, a, b}=\tpi_n \fp_{n, a, b}+\cO\left(\theta^N \right) .$$
From  \eqref{P-Zetas}, Lemma \ref{LmCurrent} and \eqref{RhoMNorm} we get
$$ \tpi_n \fp_{n, a, b}=\frac{m}{\rho_n \one }
\left(\frac{\rho_n Q_n ( \fm_{n-1}-\zeta_{n-1}\fm_n )}{\cM_n-\cM_a}+
\frac{\rho_nP_n( \fm_{n+1} - \zeta_{n+1}^-\fm_n )}{\cM_b-\cM_n}+\cO\left(\frac{1}{d^2}\right)\right)$$
$$=
\frac{1}{2\rho_n \one}\left(\frac{1}{\cM_n-\cM_a}+
\frac{1}{\cM_b-\cM_n}+\cO\left(\frac{1}{d^2}\right)\right). $$

It follows that
$\tpi_n \fp_{n, a, b} \bdeta_n$ has asymptotically exponential distribution with parameter 1 and hence
$$ \tG_{a,b}((n,i); n)=\frac{1}{\tpi_n \fp_{n, a, b}}+\cO(1)=\frac{2(\cM_n-\cM_a)(\cM_b-\cM_n)}{ (\cM_b-\cM_a)}
(\rho_n \one) +\cO(1). $$
Next for $k<n$ let
$$\cP_{a}((k,i); (n,j))=\mathbb{P}(\xi \text{ reaches }\bbL_n \text{ at }(n,j) \text{ before reaching }
\bbL_a|\xi(0)=(k,i)). $$
Then
$$ \tG_{a,b}((k,i); n)
=\sum_j \cP_{a}((k,i); (n,j))
\tG_{a,b}((n,j);n)$$
$$=\left[\frac{2(\cM_n-\cM_a)(\cM_b-\cM_n)}{(\cM_b-\cM_a)}
(\rho_n \one)+\cO(1)\right]
\sum_j \cP_{a}((k,i); (n,j))
$$
 Note that
$$ \cP_{a}((k,i);n):=\sum_j \cP_{a}((k,i); (n,j))=
\mathbb{P}(\xi \text{ reaches }\bbL_n  \text{ before reaching }
\bbL_a|\xi(0)=(k,i)). $$
Applying again the Optional Stopping Theorem to the stopping time
$\fs$ which is the first time the walker reaches either $\bbL_a$ or $\bbL_b$ we  we get
$$ (\cM_n+\cO(1))\cP_a((k,i);n)+(\cM_a+\cO(1))(1-\cP_a((k,i);n))=\cM_k+\cO(1). $$
Hence
$$ \cP_{a}((k,i); n)=\frac{\cM_k-\cM_a}{\cM_n-\cM_a}+\cO\left(\frac{1}{d}\right). $$
This proves \eqref{GFLayer} for $k<n.$ The case $k>n$ is analyzed similarly.
\end{proof}

\begin{remark}
\eqref{OST-Prob} also shows that there is a constant $c$ such that for each $y\in \{1,\dots, m\}$
\begin{equation}
\label{LE-Lower}
\Prob(\bdeta_n=1|\xi(0)=(n,y))\leq \frac{c}{d}.
\end{equation}
This bound will be useful in Section \ref{ScCLT}.
\end{remark}

\section{Environment viewed by the particle: the Law of Large Numbers.}
\label{ScEnvLLN}
From now on we consider only those environments which, in addition to
\eqref{BPNorm}, \eqref{XMNorm}, \eqref{RhoMNorm},
satisfy the following assumption: there exists a constant $a$ such that
\begin{equation}
\label{AvOcc}
\lim_{N\to\pm\infty} \frac{1}{|N|} \sum_{n=0}^{N-1} \rho_n \one =a.
\end{equation}
Examples \ref{ExQP}--\ref{ExPert} satisfy \eqref{AvOcc}. In fact, in Section \ref{SSCLTEx}
we prove a stronger result \eqref{RC-Occ-V}.

Let $h:\bbS\to\reals$ be a bounded function and $h_n$ be a sequence of vectors with components
$h_n(i)=h(n,i).$
Let
\begin{equation}
\label{HNSum}
H_N=\sum_{n=0}^{N-1} h(\xi(n)).
\end{equation}
In this section we establish the following result.

\begin{lemma}
\label{LmEnvLLN}
Suppose that $h_n$ is such that
\begin{equation}
\label{AvH}
\lim_{N\to\pm\infty} \frac{1}{|N|} \sum_{n=0}^{N-1}\rho_n h_n =\fh
\end{equation}
for some constant $\fh.$
Then $\frac{H_N}{N}$ converges in probability, as $N\to\infty$, to $\frac{\fh}{a}.$
\end{lemma}

\begin{remark}
\label{RemExt}
Assumptions \eqref{AvOcc} and \eqref{AvH} imply that $h$ is an extensive observable,
that is there exists the finite volume limit
$$ \lim_{R\to\pm \infty} \frac{\int_{\bbS_R} h(z) d\mu(z)}{\mu(\bbS_R)}=\frac{\fh}{a}$$
where $\bbS_R$ is the set of points in the strip $\bbS$ such that the $x$ coordinate is
between $0$ and $R$ and $\mu$ is the invariant measure for our walk: for $A\in \bbS$
$$ \mu(A)=\sum_{z\in A} \rho(z). $$
We refer the reader to \cite{Len} for a discussion of the ergodic properties of extensive observables.
\end{remark}

A typical application of Lemma \ref{LmEnvLLN} is the following. Suppose that the environment
$\omega = ((P_n, Q_n, R_n))_{n=-\infty}^{\infty}$ is as in Remark \ref{Rem2.3}
and set $h_n=1_{(T^n \omega, Y_n) \in \cA}\one$,
where $\cA\subset \Omega\times \{1,\dots, m\}$ (this defines a function $h$ - see
Section \ref{notations}). Then $\frac{H_N}{N}$ describes
how often the walker sees the environment from $\cA.$ For example, one can ask how often the drift
or the variance of the walker's increment are
of a certain size.
Quite often the law of large numbers for $H_N$ is obtained as a consequence
of ergodicity of the {\em environment viewed by the particle} process, see e.g. \cite{BS}. This approach, however,
makes it difficult to control the exceptional zero measure set in the ergodic theorem.
In this section we present a different argument which allows one to obtain explicit sufficient conditions
for the law of large numbers (namely, \eqref{AvH}).

\begin{proof}[Proof of Lemma \ref{LmEnvLLN}]
Let us first describe the idea of the proof. Fix $\eps>0.$ We need to show that
$\DS H_N-\left(\frac{\fh}{a}-\eps\right) N$ is positive for large $N$ while
$\DS H_N-\left(\frac{\fh}{a}+\eps\right) N$ is negative for large $N$ with probability close to 1.
To this end we divide the sum \eqref{HNSum} into blocks.
Choose a small constant $\delta$ (the exact requirements on $\delta$  will be
explained later, see the proof of \eqref{NumRen} below) and let $L_N=\lfloor\delta \sqrt{N}\rfloor.$
We will consider our random walk only at the moments when
 $X(t)$ visits the nodes of the lattice $L_N\integers,$ more precisely, when $X$ moves from
 one node to the next. That is,
 define $\tau_0=0,$ and for $k>0$ let
\begin{equation}
\label{DefTauK}
\tau_k=
\min(j\geq \tau_{k-1}: X_j=X_{\tau_{k-1}}-L_N\text{ or } X_j=X_{\tau_{k-1}}+L_N).
\end{equation}
We would like to use the results of Section \ref{ScGreen} to show
$\DS H_{(k, \eps)}^+:=\sum_{n=1}^{\tau_k} \left[h(\xi(n))-\left(\frac{\fh}{a}-\eps\right)\right]$
is a submartingale and
$\DS H_{(k, \eps)}^-:=\sum_{n=1}^{\tau_k} \left[h(\xi(n))-\left(\frac{\fh}{a}+\eps\right)\right]$
is a supermartingale with respect to the natural filtration and then use the large deviation estimates
for supermartgales from Appendnix \ref{AppLargeMod}.
However, for a fixed $\delta$, \eqref{XMNorm} and \eqref{AvH} only allow us to control
the nodes of $L_N  \integers$ which are not too far from the origin, so an additional cut off is required.

Using the maximal inequality for martingales we can find a constant $K$ such that
$$ \Prob\left(\max_{t\in [0, N]} |\fm(\xi(t))|\geq \frac{K\sqrt{N}}{2}\right)\leq \frac{\eps}{2}. $$
Now \eqref{XMNorm2} gives
\begin{equation}
\label{ChoiceK}
 \Prob\left(\max_{t\in [0, N]} |X(t)|\geq K\sqrt{N}\right)\leq \frac{\eps}{2}.
\end{equation}

Denoting $ a_k=X_{\tau_{k-1}}-L_N,$ $b_k=X_{\tau_{k-1}}+L_N $ we have
\begin{equation}
\label{AvEscape}
\EXP(\tau_{k}-\tau_{k-1}|\xi_{\tau_{k-1}})=
\sum_{n=a_k}^{b_k}
\sum_{j=1}^m G_{a_k, b_k}
(\xi_{\tau_{k-1}}; (n, j)).
\end{equation}
We claim that for each $K$ we have
\begin{equation}
\label{AvEscape2}
 \EXP(\tau_{k}-\tau_{k-1}|\xi_{\tau_{k-1}})
\sim a L_N^2
\end{equation}
provided that $N$ is large enough and $|X_{\tau_{k-1}}|\leq K\sqrt{N}.$

To prove \eqref{AvEscape2}
divide the segment $[a_k, b_k]$ into subsegments
$I_j=[s_j, s_{j+1}]$ of length $[\tdelta \sqrt{N}]$ where $\tdelta\ll \delta.$
\eqref{XMNorm}, Lemma \ref{LmGF}, and \eqref{AvOcc} show that
the contribution to \eqref{AvEscape} of terms with
$n\in I_j$ is
asymptotic to
$$ \tdelta_N  a \cG(0, s_j-X_{\tau_{k-1}}; -L_N, L_N). $$
Summing over the intervals $I_j$ we obtain \eqref{AvEscape2}.

Let
\begin{equation}
\label{DefHK}
\hat k=a^{-1} \delta^{-2},
\end{equation}
$T_N=\tau_{\hat k}.$
We claim that if $\delta$ is sufficiently small then
\begin{equation}
\label{NumRen}
\Prob\left(\left|\frac{T_N}{N}- 1\right|>\eps\right)<\eps.
\end{equation}
Indeed define a sequence $\ttau_k$ such that $\ttau_0=0$ and
$$ \ttau_{k}-\ttau_{k-1}=
\begin{cases}  \tau_{k}-\tau_{k-1} & \text{if } |\xi_{\tau_{k-1}}|\leq K \sqrt{N} \\
a L_N^2 & \text{otherwise.} \end{cases} $$
We want to estimate
$\Prob(\ttau_{\hat k}\geq (1+\eps) N).$ To this end we apply Proposition \ref{PrLLNMart}
from Appendix \ref{AppLargeMod}
with $$\Delta_k=\frac{(\ttau_k-\ttau_{k-1})}{a L_N^2}-(1+\eps).$$
To apply this proposition we need to check conditions \eqref{SubMart} and \eqref{Cramer}.
For the case at  hand, \eqref{SubMart} follows from \eqref{AvEscape2}. To prove \eqref{Cramer} we use
that there exists a constant $\theta<1$ such that
for each $K\in\reals$ there is a constant $N_0=N_0(K)$
such that if $N\geq N_0$ and $|\xi_{\tau_{k-1}}|\leq K\sqrt{N}$ then
for all $l\in\naturals$
\begin{equation}
\label{ExpTail}
\Prob(\tau_{k}-\tau_{k-1}>2 a l L_N^2)<\theta^l.
\end{equation}
Indeed similarly to \eqref{AvEscape2} one can show that if $N\geq N_0(K)$ then for each $(x, y)$ such that
$|x-x(\xi_{\tau_{k-1}})|\leq L_N$ for each $s\in \integers$ we have
$$ \EXP(\tau_k-s|\xi(s)=(x,y))\leq 1.1 a L_N^2.$$
Combining this with the Markov inequality we see that for any stopping time
$\fs$
\begin{equation}
\label{ProbEscape}
 \Prob(\tau_k>\fs+2a L_N^2| \tau_k\geq \fs)\leq \frac{1.1}{2}.
 \end{equation}
 Applying \eqref{ProbEscape} with $\fs=\tau_{k-1}$ we obtain
 \eqref{ExpTail} with $l=1.$
\eqref{ExpTail} for $l>1$ follows
by induction on $l$ by applying \eqref{ProbEscape} with $\fs=\tau_{k-1}+2a (l-1).$

Now Proposition \ref{PrLLNMart} gives
$$ \Prob(\ttau_{\hat k}\geq (1+\eps) N)\leq e^{-\brc\sqrt{\eps \hk}}. $$
Likewise
$$ \Prob(\ttau_{\hat k}\leq (1-\eps) N)\leq e^{-\brc \sqrt{\eps\hk}} . $$
Combining the last two displays with \eqref{DefHK} we see that for large $N$
\begin{equation}
\label{LLN-TTau}
\Prob\left(\left|\frac{\ttau_{\hat k}}{N}- 1\right|>\eps\right)<\frac{\eps}{2}.
\end{equation}
By our choice of $K$ (see \eqref{ChoiceK}), for large $N$ we have
\begin{equation}
\label{TauT}
\Prob(\ttau_{\hat k}\neq T_N)<\frac{\eps}{2}.
\end{equation}
Combining \eqref{LLN-TTau} and \eqref{TauT}
we obtain \eqref{NumRen}.

Next, similarly to \eqref{AvEscape2} we get
\begin{equation}
\label{GH}
\EXP(H_{\tau_{k}}-H_{\tau_{k-1}}|\xi_{\tau_{k-1}})\sim \frac{\fh L_N^2}{2}
\end{equation}
and similarly to \eqref{NumRen} we get (possibly, after decreasing $\delta$) that
$$ \Prob\left(\left|\frac{H_{\tau_\brk}}{N}-\frac{\fh}{a}\right|>\eps\right)<\eps. $$
Indeed we can apply Proposition \ref{PrLLNMart} since
\eqref{SubMart} follows by
\eqref{GH} while \eqref{Cramer} follows from \eqref{ExpTail} since
$h$ is bounded so for some constant $C$
$$| H_{\tau_k}-H_{\tau_{k-1}}|\leq C (\tau_{k}-\tau_{k-1}). $$
Also since $h$ is bounded,
$\DS \frac{|H_{\tau_\brk}-H_N|}{N}\leq C\left|\frac{\tau_\brk}{N}-1\right|$ and so
\eqref{LLN-TTau} and \eqref{TauT} give
\begin{equation}
\label{HOvershoot}
\Prob\left(\frac{|H_{\tau_\brk}-H_N|}{N}> C \eps\right)<\eps.
\end{equation}
Since $\eps$ is arbitrary, \eqref{GH} and \eqref{HOvershoot} prove the lemma.
\end{proof}

\begin{remark}
We note that
the information on $X_{\tau_k}$ obtained in the proof of Lemma~\ref{LmEnvLLN},
especially \eqref{GH}, will play a crucial
role in the sequel. In particular, it will be used in Section \ref{ScCLT} to show that,
under appropriate assumptions, $X_{\tau_k}/L_N$ is well approximated by the simple
random walk. Passing to the limit as $N\to\infty,$ $\delta_N\to 0$ we shall obtain the CLT for $X(t).$
\end{remark}

\section{The Central Limit Theorem}
\label{ScCLT}
\subsection{Sufficient conditions for the CLT}
In this section, with a slight abuse of notation,
we write $\xi_{Nt}$,  $X_{Nt}$, and $Y_{Nt}$ for $\xi_{\lfloor Nt\rfloor},$
$X_{\lfloor Nt\rfloor}$, $Y_{\lfloor Nt\rfloor}$ respectively.

Denote $W_N(t)=\frac{X_{Nt}}{\sqrt{N}},$ where $t\in [0,\,1]$. Let $\fq_n$ be a column vector with components
\[
\begin{aligned}
\fq_n(i)&=\EXP\left((\fm(\xi_{k+1})-\fm(\xi_{k}))^2\big|\xi_k=(n, i)\right)\\
&=\sum_{j'\in\{-1,0,1\}, 1\le i'\le m}\mathfrak{P}\left((n,i),(n+j',i')\right)
\left(\fm_{n+j'}(i')-\fm_{n}(i)\right)^2,
\end{aligned}
\]
where $\mathfrak{P}(\cdot,\cdot)$
are the transition probabilities \eqref{striptransition} and $\fm_n(i)$, $\fm(\xi_{k})$, etc are
as in \eqref{MartEq} and Remark \ref{notations}.
\begin{theorem}
\label{ThCLTWalk}
If \eqref{AvOcc} holds and there is a constant $b$ such that
\begin{equation}
\label{AvQV}
\lim_{N\to\pm\infty} \frac{1}{|N|} \sum_{n=0}^{N-1}\rho_n \fq_n  =b
\end{equation}
then
$W_N(t)$ converges in law as $N\to\infty$ to $\cW(t)$-the Brownian Motion with zero mean and variance $Dt$,
where $D=\frac{b}{a}$ with $b$ as in \eqref{AvQV} and $a$ as in \eqref{AvOcc}.
\end{theorem}

\begin{proof}
In view of \eqref{XMNorm2} it suffices to show that
\begin{equation}
\label{MartToBM}
\hW_N(t) \Rightarrow \cW
\end{equation}
where
\begin{equation}
\label{DefhW}
 \hW_N(t)=\frac{\fm({\xi_{Nt}})}{\sqrt{N}}.
\end{equation}
Let $\fQ_N=\sum_{n=0}^{N-1} \fq(\xi_n),$ where $\fq(\xi_n)=\fq_{X_n}(Y_n)$. By \cite[Theorem 3]{Bn} to prove \eqref{MartToBM} it suffices to check that
\begin{equation}
\label{LinQVar}
\frac{\fQ_N}{N}\Rightarrow D,
\end{equation}
but this follows from \eqref{AvQV} and Lemma \ref{LmEnvLLN}.
\end{proof}

\begin{corollary}
For uniquely ergodic environments with bounded potential the Central Limit Theorem holds
for {\bf all} $\omega.$
\end{corollary}

In \cite{DG4}, the Central Limit Theorem was proved for almost all $\omega$
for a wide class of environments which includes the uniquely ergodic ones as a particular case.
Here, for uniquely ergodic environments, we prove that this result holds for all (rather than almost all)
$\omega.$

\subsection{Expectation of the local time.}
Here we discuss the distribution of the local time of the walk.
Let $V((k, y), N)$ be the number of visits to the site $(k,y)$ by our walk before time $N.$

\begin{lemma}
  \label{LmUILT}
Under the assumptions of Section \ref{ScGreen}
for each $K>0$ the collection of random variables
$$\left\{\frac{V((k, y), N)|\xi(0)=z)}{\sqrt{N}}\right\}$$
is uniformly integrable where the uniformity is
with respect to $N\in \naturals,$
$z\in \bbS,$ and $(k, y)\in \bbS$ such that $|k|\leq K\sqrt{N}.$
\end{lemma}

In the proof we will use the following notion. Let $\cX$ and $\cY$ be non-negative random variables.
We say that $\cY$ {\em stochastically dominates } $\cX$ if for each $t>0$
$ \Prob(\cX\geq t)\leq \Prob(\cY\geq t). $ Clearly if $\cY$ stochastically dominates $\cX$ then
$\EXP(\cY)\geq\EXP(\cX).$

\begin{proof}[Proof of Lemma \ref{LmUILT}]
It suffices to prove the result for the walk starting from $z=(k, y)$ since the local time does not accumulate
before the first visit to the site $(k, y).$

By \eqref{AvEscape2} and the maximal inequality for martingales, there is a constant $\hp<1$ such that
for each $K$ there exists $N_0(K)$ such that if $N\geq N_0(K)$ then for any
$(k', y')\in \bbS$ with $|k'|\leq (K+1)\sqrt{N}$, the probability that the random walk exits the segment
$[k'-\sqrt{N}, k'+\sqrt{N}]$ before time $N$ is less than $\hp$
(In fact, $\hp$ can be any number which is greater than the probability that the Brownian motion with
zero mean and with variance $Dt$ exits the interval $[-1, 1]$ before time 1).

Let $\eta$ be the total number of visits to $(k, y)$ before the walk exits from the segment
$(k-\sqrt{N}, k+\sqrt{N}).$ By the foregoing discussion,
the probability that $V((k, y), N)\leq \eta$ is greater than $1-\hp.$
Therefore, for large $N,$
$V((k, y), N)$ is stochastically dominated by $\eta+\hp V((k, y), N).$
Iterating this estimate we conclude that $V((k, y), N)$ is stochastically dominated by
$\DS \cV:=\sum_{r=1}^\hG \eta_r$ where $\hG$ is has geometric distribution with parameter
$1-\hp$ and $\eta_r$ are i.i.d random variables independent of $\hG$ and having the
same distribution as $\eta.$
Since $\DS \EXP(\cV)=\frac{\bE(\eta)}{1-\hp}$ it suffices to show the uniform integrability of $\eta/\sqrt{N}$
(with respect to time and the initial position of the walk).
However the fact that $\{\eta/\sqrt{N}\}$ is uniformly integrable follows from \eqref{LE-Lower}
\end{proof}

Let $\fl_{x,t}$ denote the local time of the standard Brownian motion.
\begin{theorem}
\label{ThConvLT}
 Suppose that \eqref{AvOcc} and \eqref{AvQV} hold. Let $(k_N, y_N)\in \bbS$ be a sequences such that
 $\frac{k_N}{\sqrt{N}}\to x$ as $N\to\infty$. Then, as $N\to\infty$
$$ \frac{V((k_N, y_N), N)}{\rho_{k_N}(y_N) \sqrt{N}}\Rightarrow \fl_{x, 1/a}. $$
\end{theorem}

Combining Theorem \ref{ThConvLT} with Lemma \ref{LmUILT} we obtain

\begin{corollary}
\label{CrExpVis}
Suppose that $(k_N,y_N)$ is a sequence of points in $\bbS$ such that
$\frac{k_N}{\sqrt{N}}\to x.$ Then uniformly for $x$
in a compact set we have
\begin{equation}
\label{ConvLT-Cl}
 \lim_{N\to\infty} \EXP\left(\frac{V((k_N, y_N), N)}{\rho_{k_N} (y_N)\; \sqrt{N}}\right)
=\bE(\fl_{x, 1/a}).
\end{equation}
\end{corollary}

\begin{proof}[Proof of the theorem]
  Consider first the case $k_N\equiv 0.$ We use the same notation as in the proof of Lemma \ref{LmEnvLLN}.
  In particular we let $L_N=\lfloor\delta\sqrt{N}\rfloor$ for a small constant $\delta.$

  Fix $\eps>0.$ We show that if $\delta$ is sufficiently small then for large $N$
  the following estimates hold:
\begin{equation}
\label{LTMeso}
\max_{u\in\reals} \left|\Prob\left(\frac{V((0, y_N), \tau_\hk)}{\sqrt{N}}\leq u\right)
-\bP\left(\rho(0, y_N) \, \fl_{0, 1/a}\leq u\right)\right| \leq \eps,
\end{equation}
where $\hk$ is defined by \eqref{DefHK}, $\tau_k$ is defined by \eqref{DefTauK}
and
\begin{equation}
\label{LTEdge}
\Prob\left(\frac{|V((0, y_N), N)-V((0, y_N), \tau_\hk)|}{\sqrt{N}}>\sqrt\eps\right)\leq C \sqrt{\eps}.
\end{equation}

To prove \eqref{LTEdge} we note that by \eqref{NumRen}, if $N$ is large enough, then
$\Prob(|\tau_\hk-N|\geq \eps N)\leq \eps.$
On the other hand if $|\tau_\hk-N|\leq \eps N$ then
$$ |V((0, y_N), N)-V((0, y_N), \tau_\hk)|\leq V((0, y_N), (1+\eps)N)-V((0, y_N), (1-\eps)N). $$
By Lemma \ref{LmUILT} the expectation of the RHS is less than $\brC\eps \sqrt{N}$ so
by the Markov inequality
\begin{equation}
\label{MInLT}
\Prob(V((0, y_N), (1+\eps)N)-V((0, y_N), (1-\eps)N)\geq \sqrt{\eps} \; \sqrt{N})\leq \brC\sqrt{\eps}.
\end{equation}
Combining \eqref{NumRen} and \eqref{MInLT} we obtain \eqref{LTEdge}.

To prove \eqref{LTMeso} let $U_j$ be the number of visits to $(0, y_N)$ during the time interval $[\tau_{j-1}, \tau_j].$
Note that $U_j=0$ unless $\xi(\tau_{j-1})\in \bbL_0.$
In case $\xi(\tau_{j-1})\in \bbL_0,$ \eqref{LE-Lower}  shows that
$$ \Prob\left(U_j=0\,|\,\xi(\tau_{j-1})\in \bbL_0\right)\leq \frac{C}{L_N}. $$
On the other hand, the general theory of Markov chains shows that, conditioned on $U_j\neq 0,$
$U_j$ has geometric distribution with the mean $G_{-L_N, L_N}((0, y_N); (0, y_N))$
and moreover it is independent of
$\xi(\tau_j).$ By Lemma \ref{LmGF}
$$G_{-L_N, L_N}((0, y_N); (0, y_N))=L_N \rho(0, y_N)(1+o_{N\to\infty} (1)). $$
Now it is easy to show using, for example, Proposition \ref{PrLLNMart}, that
$$ \Prob\left(\left|V((0, y_N), \tau_\hk)-L_N \fn(\hk) \rho(0, y_N)|\right)\geq \frac{\eps \sqrt{N}}{3}\right)\to 0 \text{ as } N\to \infty $$
where $\fn(k)=\Card(j<k: \xi(\tau_j)\in \bbL_0).$

Since the local time of the simple random walk converges after the diffusive rescaling to
a local time of the Brownian Motion (\cite{Bor}), we can take $\delta$ so small that
$$ \max_{u\in\reals} \left|\bP(\delta \tilde\fn(\hk)\leq u)-\bP(\fl_{0, 1/a}\leq u)\right|\leq \frac{\eps}{3} $$
where $\tilde\fn(\hk)$ is the number of times the simple symmetric random walk returns to 0 before time $\hk.$
On the other hand, \eqref{XMNorm} and the Optional Stopping Theorem for martingales show that
$\left\{\frac{\xi(\tau_j)}{L_N}\right\}_{j\in\naturals}$ converges as $N\to\infty$ to the simple random walk on $\integers.$
Hence for each $\delta$ we have
$$ \max_{u\in\reals} \left|\bP(\delta \tilde\fn(\hk)\leq u)-\bP(\delta\fn(\hk)\leq u)\right|\leq \frac{\eps}{3} $$
provided that $N$ is large enough. Combining the last three displays we obtain
\eqref{LTMeso}.

This completes the proof of the Theorem in the case $k_N\equiv 0.$ The same argument shows that
for each $k_N, y_N, y_N'$, if the walk starts from $(k_N, y_N')$ then
$ \frac{V((k_N, y_N), N)}{\sqrt{N} \rho(k_N, y_N)}$ converges to $\fl_{0, 1/a}.$  Let $\ft_k$
be the first time the walk reaches layer $\bbL_k.$ Divide
$[0,1]$ into intervals $I_j$ of small length $h$ and let $t_j$ be the center of $I_j.$
By Theorem \ref{ThCLTWalk}, the probability that
$\frac{\ft_{k_N}}{N}\in I_j$ converges as $N\to\infty$ to $\bP(\cT_x\in I_j)$ where $\cT_x$ is the first
time the standard Brownian Motion reaches $x.$
On the other hand
conditioned on $\frac{\ft_N}{N}\in I_j$ we have that the distribution of
$ \frac{V((k_N, y_N), N)}{\sqrt{N}}$ is close to the distribution of $\fl_{0, (1-t_j)/a}$
(the closeness means that the error goes to $0$ when $h\to 0$ and $N\to\infty$).
Therefore for each  $s$
$$ \lim_{N\to\infty} \Prob\left(\frac{V((k_n, y_N), N)}{\sqrt{N} \rho(k_N, y_N)}\geq s\right)
=\int_0^1 f_\cT(t) \bP(\fl_{0, (1-t)/a}\geq s) dt $$
where $f_{\cT_x}$ is the density of $\cT_x.$
The last integral is equal to
$\bP(\fl_{x, 1/a}\geq s)$ completing the proof of Theorem \ref{ThConvLT}.
\end{proof}

\subsection{Rate of convergence.}
\label{SSRates}
Here we  estimate the rate of convergence in Theorem \ref{ThCLTWalk}
assuming that we have a good control of
error rates in \eqref{XMNorm}, \eqref{AvOcc}, and \eqref{AvQV}.

Let $\Phi(x)$ denote the distribution function of a standard normal random variable.

We will use the following two results.
\begin{proposition} $($\cite[Theorem 3.7]{HH}$)$
\label{Mart-BE}
Given constants $C_1, C_2, C_3$ there is a constant $C_4$ such that the following holds.
Let $Z_n$ be a martingale difference sequence such that for $n\leq N$
\begin{equation}
\label{MartExpM}
|Z_n|\leq C_1
\end{equation}
and $\cQ_N=\sum_{n=1}^N \bE(Z_n^2|\cF_{n-1})$
satisfies
\begin{equation}
\label{OscQV}
\bP\left(\left|\cQ_N-N\right|\geq C_2 \sqrt{N} \ln^2 N\right)\leq \frac{C_3 \ln N}{N^{1/4}}
\end{equation}
Then 
$$ \sup_x \left|\bP\left(\frac{\sum_{n=1}^N Z_n}{\sqrt{N}}\leq x\right)-\Phi(x)\right|\leq
\frac{C_4 \ln N}{N^{1/4}} . $$
\end{proposition}

\begin{proposition}
\label{PrConv}

Let $S, Z$ be random variables and set
$$
 \delta=\sup_x |\bP(S\leq x)-\Phi(x)|, \quad \delta^*=\sup_x |\bP(S+Z\leq x)-\Phi(x)|.$$
Then:

\noindent
(a) There exists a constant $C$ (independent of $S$ and $Z$), such that
$$ \delta^*\leq 2\delta+C ||\sqrt{\bE(Z^2|S)}||_{\infty}, $$
\noindent
$(b)\ \ \delta^*\leq \delta+\bP(Z\neq 0).$
\end{proposition}

\begin{proof}
Part (a) is proven in \cite[Lemma 1]{B}. To prove part (b) it suffices to observe that by the triangle inequality
$ \displaystyle |\delta-\delta^*|\leq  \sup_x |\bP(S\leq x)-\bP(S+Z\leq x)|. $
\end{proof}

In this section, in order to bound the error rate in the CLT, we assume that there is $\beta_1<1$ such that for each $L\geq N^{0.01}$ and each
$|k|\leq N$
\begin{equation}
\label{RCMart}
\left|\fm_{k+L}(1)-\fm_{k-L}(1)-2L\right|\leq C L^{1-\beta_1},
\end{equation}
\begin{equation}
\label{RC-Occ-V}
\left|\sum_{j=k-L}^{k+L} \rho_j \one-2L a\right|\leq C L^{1-\beta_1},
\end{equation}

\begin{equation}
\label{RC-Occ-Q}
\left|\sum_{j=k-L}^{k+L} \rho_j \fq_j-2L b\right|\leq C L^{1-\beta_1}.
\end{equation}
Recall the notation of Section \ref{ScEnvLLN}.
Define $\tau_j$ as in  \eqref{DefTauK}
with $L_N=N^{1/4}.$
Note that \eqref{RC-Occ-V}, \eqref{RC-Occ-Q}  implies that
\begin{equation}
\label{RC-EXP-Occ-V}
\EXP(\tau_{k}-\tau_{k-1}|\xi_{\tau_{k-1}})=aL_N^2\left(1+\cO\left(L_N^{-\beta_1}\right)\right),
\end{equation}
\begin{equation}
\label{RC-EXP-Occ-Q}
\EXP\left(\sum_{n=\tau_{k-1}}^{\tau_k} \fq(\xi_n) |\xi_{\tau_{k-1}}\right)=
bL_N^2\left(1+\cO\left(L_N^{-\beta_1}\right)\right)
\end{equation}
provided that $|X(\tau_{k-1})|\leq N.$

To establish \eqref{RC-EXP-Occ-V}
we temporarily denote
$$ \xi=\xi_{\tau_{k-1}}=(x,y),\quad c=X(\tau_{k-1})-L_N, \quad d=X(\tau_{k-1})+L_N .$$
Then Lemma \ref{LmGF} gives
$$ \EXP(\tau_{k}-\tau_{k-1}|\xi_{\tau_{k-1}})=
\sum_{n=c}^{d}\tG_{c, d}(\xi; n). $$
In view of \eqref{GFLayer} and \eqref{RCMart}
$$\sum_{n=c}^{d}\tG_{c, d}(\xi; n)=\left[\sum_{n=c}^d \cG(x, n, c, d) \rho_n \one \right]
+\cO\left(L_N^{2-\beta_1}\right). $$
The main term equals to
$$ \sum_{n=c}^d \cG(x, n, c, d) \rho_n \one =
\sum_{n=c}^d \cG(x, n, c, d) a +\sum_{n=c}^d \cG(x, n, c, d) (\rho_n \one-a) $$
$$=
a L_N^2+\cO\left(L_N\right)+\sum_{n=c}^d \cG(x, n, c, d) (\rho_n \one-a). $$
To estimate the last term denote $\DS \cI_n=\sum_{k=x}^n (\rho_n \one-a).$
Summation by parts gives
$$ \sum_{n=c}^d \cG(x, n, c, d) (\rho_n \one-a)=-\sum_{n=c}^d \cI_n \nabla \cG(x, n, c, d)$$
where $\nabla$ is the difference operator, $\nabla H=H_n-H_{n-1}.$
The first term in the last sum is $\cO\left(L_N^{1-\beta_1}\right)$ and the second term is
bounded. Whence the last sum is $\cO\left(L_N^{2-\beta_1}\right)$ proving
\eqref{RC-EXP-Occ-V}.
The proof of  \eqref{RC-EXP-Occ-Q} is similar.

\begin{theorem}
\label{ThCLTRate}
If \eqref{RCMart}, \eqref{RC-Occ-V}  and \eqref{RC-Occ-Q}
hold then for each $\eps>0$ there is a constant $C=C_\eps$ such that
\begin{equation}
\label{CLTRateMain}
\sup_x \left|\Prob\left(\frac{X_N}{\sqrt{DN}}\leq x\right)-\Phi(x)\right|\leq C N^{-(\upsilon-\eps)}
\end{equation}
where
$$\upsilon= \frac{1}{2} \; \min\left(\frac{1}{4},\; \beta_1 \right).$$
\end{theorem}

\begin{proof}
To establish the theorem it suffices to show that
\begin{equation}
\label{CLTRateMart}
 \sup_x \left|\Prob\left(\frac{\fm(\xi_N)}{\sqrt{DN}}\leq x\right)
 -\Phi(x)\right|\leq C N^{-\brupsilon} .
\end{equation}
where
$$\brupsilon= \frac{1}{2} \; \min\left(\frac{1}{4}-\eps,\; \beta_1 \right).$$
Indeed suppose that \eqref{CLTRateMart} holds. Let
$$ \tX_N=\begin{cases} X_N & \text{if } |\fm(\xi_N)|<N^{\frac{1+\heps}{2}} \\
                                       \fm(\xi_N)&\text{otherwise} \end{cases}$$
where $\heps$ is a sufficiently small number. Then due to \eqref{RCMart} there is a constant $K$ such that
$\left|\tX_N-\fm(\xi_N)\right|\leq K N^{\frac{(1+\heps)(1-\beta_1)}{2}}.$
Therefore
\begin{equation}
\label{Squeeze}
\Prob\left(\frac{\fm(\xi_N)}{\sqrt{DN}}\leq x-\frac{K}{\sqrt{D}} N^{\frac{(1-\beta_1)(1+\heps)-1}{2}}\right)
\leq
\Prob\left(\frac{\tX_N}{\sqrt{DN}}\leq x\right)
\end{equation}
$$ \leq \Prob\left(\frac{\fm(\xi_N)}{\sqrt{DN}}\leq x+
\frac{K}{\sqrt{D}} N^{\frac{(1-\beta_1)(1+\heps)-1}{2}}\right). \quad\quad\quad \quad\quad\quad \quad\quad\quad
$$
Combining \eqref{Squeeze} with \eqref{CLTRateMart} we obtain
\begin{equation}
\label{CLTRateTX}
\sup_x \left|\Prob\left(\frac{\tX_N}{\sqrt{DN}}\leq x\right)
 -\Phi(x)\right|\leq C N^{-\brupsilon}+\brK N^{\frac{\heps -\beta_1(1+\heps)}{2}}
 \leq \brC N^{-(\upsilon-\eps)}
\end{equation}
   provided that $\heps$ is small enough,

   On the other hand, by Azuma inequality, there are constants $\tc_1, \tc_2$ such that
   \begin{equation}
     \label{TX-X}
   \Prob(\tX_N\neq X_N)=P\left(\fm(\xi_N)>N^{\frac{1+\heps}{2}}\right)\leq c_1 e^{-c_2 N^\heps}.
   \end{equation}
Combining \eqref{CLTRateTX} with \eqref{TX-X} we obtain \eqref{CLTRateMain}.


It remains to obtain \eqref{CLTRateMart}.
Let $Z_j=\frac{\fm(\xi(\tau_j))-\fm(\xi(\tau_{j-1}))}{L_N}$
and
$$\cQ_j=\sum_{n=1}^{\tau_j} \fq(\xi(n)),
\quad j^*=\min\left(j: \cQ_{\tau_j}>DN \right),\quad \tau^*=\tau_{j^*},
\quad \cQ^*=\cQ_{\tau^*},\quad \fm^*=\fm(\xi_{\tau^*}).$$
Note that $\DS Z_n=\pm 1+\cO\left(N^{-\beta_1}\right)$
due to \eqref{RCMart} and
\begin{equation}
\label{QVarMSt}
\Prob\left(DN \leq Q^*<DN +L_N^2 \ln^2 L_N\right)\leq \theta^{\ln^2 L_N}
\end{equation}
due to \eqref{ExpTail}.

Next, we show that
if $R_1$ is a large constant then  for each $j\geq \frac{N}{10  a L_N^2}$
we have
\begin{equation}
\label{ControlQV}
 \bP\left(\left|\cQ_{\tau_j}-a L_N^2 j\right|>R_1 j L_N^{2-\beta_2} \right) \leq c_1 e^{-c_2 N^{\beta_3}}
\end{equation}
where $\beta_2=\min(\beta_1, \frac{1}{4}-\eps),$ and
$c_1, c_2,$ and  $\beta_3$ are positive constants.
We will prove that
\begin{equation}
\label{ControlQVUp}
 \bP\left(\cQ_{\tau_j}-a L_N^2 j>R_1 j L_N^{2-\beta_2} \right) \leq \brc_1 e^{-c_2 N^{\beta_3}},
\end{equation}
the estimate of
$\DS  \bP\left(\cQ_{\tau_j}-a L_N^2 j<-R_1 j L_N^{2-\beta_2} \right)$ being similar.

To prove \eqref{ControlQVUp} we apply the results of Appendix \ref{AppLargeMod}, specifically \eqref{SEBound}
with
$$\Delta_n=\frac{\sum_{k=\tau_{n-1}}^{\tau_n} \fq(\xi(k))}{L_N^2}- \frac{R_1}{2} L_N^{-\beta_2}, \quad
\eps=\frac{R_1}{2} L_N^{-\beta_2}$$
and the number of summands equal to $j.$
Observe that \eqref{SEBound} is applicable, because \eqref{SubMart} follows from \eqref{RC-EXP-Occ-Q} since $\beta_2\leq \beta_1,$
\eqref{Cramer} holds by \eqref{ExpTail} and \eqref{LargeDrift} holds because $\beta_2\leq \frac{1}{4}-\eps.$

\eqref{ControlQV} implies that
\begin{equation}
  \label{JStUpper}
\Prob\left(j^*>\frac{2}{a} \sqrt{N} \right)=\Prob\left(j^*>\frac{2N}{a L_N^2}\right)
\leq C_1 e^{-C_2 N^\tdelta}  .
\end{equation}
Let $\fz=\DS \sum_{j=1}^{\min(\frac{2}{a} \sqrt{N}, j^*)} Z_j.$
By \eqref{JStUpper}
\begin{equation}
  \label{Z-M}
  \Prob\left(\fz\neq \frac{\fm}{\sqrt{N}}\right)
  \leq C_1 e^{-C_2 N^\tdelta}  .
  \end{equation}
\eqref{Z-M} and \eqref{QVarMSt} allow us to
apply Proposition \ref{Mart-BE}  to $\fz$
obtaining
$$ \sup_x \left|\Prob\left(\fz \leq x\right)-\Phi(x)\right|\leq \frac{C \ln N}{N^{1/8}} $$
(note that  $N^{1/8}$ appears in the denominator since we apply the proposition with
$\frac{2}{a} \sqrt{N}$ instead of $N$). Using \eqref{Z-M} once more we get
\begin{equation}
\label{CLTRateStopMart}
\sup_x \left|\Prob\left(\frac{\fm^*}{\sqrt{DN}}\leq x\right)-\Phi(x)\right|\leq \frac{C \ln N}{N^{1/8}} .
\end{equation}

Next, similarly to \eqref{ControlQV}, one can show that there is a constant $C_2$ such that
for each $j\geq \frac{N}{10  b L_N^2}$
we have
\begin{equation}
\label{ControlOcc}
 \bP\left(\left|\tau_j-b L_N^2 j\right|>R_2 j L_N^{2-\beta_2} \right) \leq c_3 e^{-c_4 N^{\beta_3}}
\end{equation}

Combining \eqref{ControlQV} with \eqref{ControlOcc} we conclude that for sufficiently large
$R_3$
\begin{equation}
\label{VarTime}
 \Prob\left(\left|\tau^*-N\right|\geq \frac{R_3 N}{L_N^{\beta_2}}  \right)\leq c_5 e^{-c_6 L_N^{\beta_3}}.
\end{equation}
Letting
$$ \tfm=\begin{cases} \fm(\xi_N) & \text{if } \left|\tau^*-N\right|\leq \frac{N}{L_N^{\beta_2}} \\
\fm^* & \text{otherwise} \end{cases} $$
we get that with probability 1
$$ \EXP\left(\left(\tfm-\fm^*\right)^2|\fm^*\right)\leq \frac{R_4 N}{L_N^{\beta_2}} $$
or, equivalently,
\begin{equation}
\label{TM-MScaled}
\EXP\left(\left(\frac{\tfm-\fm^*}{\sqrt{N}} \right)^2\Big|\frac{\fm^*}{\sqrt{N}}\right)\leq \frac{R_4}{L_N^{\beta_2}}.
\end{equation}
Therefore combining Proposition \ref{PrConv}(a) and \eqref{CLTRateStopMart}
we obtain
$$ \sup_x \left|\Prob\left(\frac{\tfm}{\sqrt{DN}}\leq x\right)-\Phi(x)\right|\leq
C L_N^{-\brupsilon}
$$
(note that $\brupsilon=\frac{\beta_2}{2}<\frac{1}{8}$, so the main contribution to the error comes from \eqref{TM-MScaled} rather than from
\eqref{CLTRateStopMart}).

Next, \eqref{VarTime} shows that
$$ \Prob(\tfm\neq \fm(\xi_N)) \leq c_5 e^{-c_6 L_N^{\beta_3}}. $$
\eqref{CLTRateMart}
follows from the last two displays and Proposition \ref{PrConv}(b).
\end{proof}

\subsection{Examples.}
\label{SSCLTEx}
Here we show that the examples of Section \ref{ScEx} satisfy \eqref{RCMart},
\eqref{RC-Occ-V}, and \eqref{RC-Occ-Q}.
It is convenient to denote $\Delta_n=\fm_n-\fm_{n-1}.$

We begin with quasiperiodic systems from Example \ref{ExQP}.
\begin{proposition}
For quasiperiodic environments
of Example \ref{ExQP}
if $\gamma$ is Diophantine then \eqref{RCMart}, \eqref{RC-Occ-V}, and \eqref{RC-Occ-Q} hold.
\end{proposition}

\begin{proof}
It is proven in \cite{DG4} that for quasiperiodic environments  with Diophantine frequency $\gamma$
$$ \Delta_n=\bDelta(\omega+n\gamma), \quad \rho_n=\brho(\omega+n\gamma) $$
where $\bDelta, \brho:\Tor^d\to\reals$ are continuous functions.
In Appendix \ref{AppReg} of the present paper we obtain a stronger result.

\begin{lemma}
\label{LmSmooth}
$\bDelta, \brho$ are $C^\infty.$
\end{lemma}
Lemma \ref{LmSmooth} implies \eqref{RCMart}, \eqref{RC-Occ-V}, and \eqref{RC-Occ-Q}
with $\beta_1=1.$
For example to check \eqref{RCMart} we use the fact that for Diophantine $\gamma$
there is a constant $c$ and
a function $u$ such that
$$ \bDelta(\omega)\one=c+u(\omega+\gamma)-u(\omega). $$
It follows that
$$ \cM_{k+L}-\cM_{k-L}=2Lc+u(\omega+L\gamma)-u(\omega-L\gamma). $$
Now \eqref{XMNorm} implies that $c=m$ proving \eqref{RCMart}.
Estimates \eqref{RC-Occ-V} and \eqref{RC-Occ-Q} are verified similarly.
\end{proof}
Since quasiperiodic environments  satisfy \eqref{RCMart}, \eqref{RC-Occ-V}, and \eqref{RC-Occ-Q} with
$\beta_1=1,$ Theorem \ref{ThCLTRate} holds for those environments with $\upsilon=\frac{1}{8}.$

\smallskip
Next, we consider independent environments from Example~\ref{ExInd}.

\begin{proposition}
\label{PrMIM-Ind}
\eqref{RCMart} \eqref{RC-Occ-V}, and \eqref{RC-Occ-Q}
hold for independent environments.
\end{proposition}

\begin{proof}
Let $\bbF_{a,b}$ be the $\sigma$ algebra generated by $\{(P, Q, R)_n\}_{a\leq n\leq b}.$
We use the following fact from Appendix \ref{AppReg}.

\begin{lemma}
\label{LmHolder}
$\rho_n=\brho(T^n \omega)$ and $\Delta_n=\bDelta(T^n \omega)$
where
$\brho:\Omega\to\reals^m$ is Holder continuous with respect to the metric $\bd$
defined by \eqref{DefD}.
\end{lemma}
By Lemma \ref{LmHolder}, there is $\theta<1$ such that for each $l$ there is $\bbF_{-l, l}$ measurable random vector $\brho^{(l)}$ such that
$|\brho(\omega)-\brho^{(l)}(\omega)|\leq \theta^l.$
Hence
$$ \left|\mathrm{E}(\rho_n \one|\cF_{-\infty, n-l})-\mathrm{E}(\brho\one )\right|\leq C \theta^l. $$
Now \cite{G} tells us that for almost every $\omega$ there exists $N_0=N_0(\omega)$
such that for all $|k|<N$ for all $L>N^{0.01}$ we have
\begin{equation}
\label{
GaposhIneq}
 \left|\sum_{n=k-L}^{k+L} \rho_n \one-2La\right|\leq \sqrt{L\ln^3 L} .
\end{equation}
This proves \eqref{RC-Occ-V}.  Estimates \eqref{RCMart} and \eqref{RC-Occ-V}
can be established similarly.
\end{proof}
The foregoing discussion shows that \eqref{RCMart} \eqref{RC-Occ-V}, and \eqref{RC-Occ-Q} hold with
$\beta_1=\frac{1}{2}-\eps$ (cf. \eqref{GaposhIneq}). Accordingly, Theorem \ref{ThCLTRate} holds with
$\upsilon=\frac{1}{8}.$

\smallskip
Finally we consider small perturbations of the simple random walk on $\integers$ from Example~\ref{ExPert}.

Then the invariant measure equation \eqref{RhoEq} reduces to a zero flux condition
(see e.g. \cite[\S 5.5]{Du})
$$ p_n \rho_n=q_{n+1} \rho_{n+1} $$
which gives
$$ \frac{\rho_{n+1}}{\rho_{n}}=\frac{1-2 a_n}{1+2 a_{n+1}} .$$
Considering first the case $n>0$ we obtain
$$ \rho_n= \rho_0 \left[\prod_{j=0}^{n-1} \left(\frac{1-2 a_j}{1+2 a_{j+1}} \right)\right].$$
Therefore the limit $\DS \rho_+=\lim_{n\to\infty} \rho_n$ exists and

\begin{equation}
\label{Rho+}
 \rho_n=\rho_+ 
+O\left(\frac{1}{n^{\kappa-1}} \right).
\end{equation}
Likewise the limit $\DS \rho_-=\lim_{n\to\infty} \rho_{-n}$ exists and
\begin{equation}
\label{Rho-}
 \rho_{-n}=\rho_-
+O\left(\frac{1}{n^{\kappa-1}} \right).
\end{equation}

Next, recall a formula for $\fm_n$ (\cite{Du}).
Let $\Delta_n=\fm_{n+1}-\fm_n.$
Then
$$ \Delta_{n+1}=\Delta_n \frac{1+2 a_n}{1-2a_n} $$
Thus the limit $\DS \Delta_+=\lim_{n\to\infty} \Delta_n$ exists and
\begin{equation}
\label{Delta+}
 \Delta_n=\Delta_0 \left[\prod_{j=0}^{n-1} \left(\frac{1+2 a_j}{1-2 a_{j}} \right)\right]=\Delta_+
 +O\left(\frac{1}{n^{\kappa-1}} \right) .
\end{equation}
Likewise the limit $\DS \Delta_-=\lim_{n\to\infty} \Delta_{-n}$ exists and
\begin{equation}
\label{Delta-}
 \Delta_{-n}=\Delta_0 \left[\prod_{j=0}^{n-1} \left(\frac{1-2 a_{-j}}{1+2 a_{-j}} \right)\right]=\Delta_-
 +O\left(\frac{1}{n^{\kappa-1}} \right) .
\end{equation}
Accordingly \eqref{NNUnbiased} is equivalent to the condition $\Delta_+=\Delta_-.$
Hence if \eqref{NNUnbiased} holds we can normalize $\{\Delta_n\}$ in such a way that
$\DS \lim_{n\to\pm \infty} \Delta_n=1.$
In this case \eqref{XMNorm} holds and \eqref{RhoMNorm} gives
$\DS \lim_{n\to\pm \infty} \rho_n=1.$

Now \eqref{Delta+}, \eqref{Delta-}, \eqref{Rho+}, and \eqref{Rho-}
show that
\begin{equation}
\label{RM-SRW}
 \Delta_n=1+O\left(\frac{1}{|n|^{\kappa-1}}\right) \quad \text{and}\quad
\rho_n=1+O\left(\frac{1}{|n|^{\kappa-1}}\right).
\end{equation}
It follows that \eqref{RCMart}, \eqref{RC-Occ-V} and \eqref{RC-Occ-Q} hold with $\beta_1=\min(\kappa-1, 1).$
Hence Theorem~\ref{ThCLTRate} holds in Example \ref{ExPert} with
$\DS \upsilon=\min\left(\frac{\kappa-1}{2}, \frac{1}{8}\right). $

\section{Different growth rates.}
\label{ScSkew}
\subsection{Notation}

In this section we consider the case where $\fm, \rho,$ and $\fq$ have different growth rates at
$-\infty$ and $+\infty.$ Thus we assume that instead of \eqref{XMNorm}, \eqref{AvOcc} and \eqref{AvQV}
we have
\begin{equation}
\label{XMNormPM}
\lim_{n\to -\infty} \frac{\cM_n}{n}=m\mu_-, \quad
\lim_{n\to +\infty} \frac{\cM_n}{n}=m\mu_+;
\end{equation}
\begin{equation}
\label{AvOccPM}
\lim_{N\to\infty} \frac{1}{N} \sum_{n=-N+1}^{0} \rho_n \one =a_-, \quad
\lim_{N\to\infty} \frac{1}{N} \sum_{n=0}^{N-1} \rho_n \one =a_+;
\end{equation}

\begin{equation}
\label{AvQVPM}
\lim_{N\to\infty} \frac{1}{N} \sum_{n=-N+1}^{0} \rho_n \fq_n  =b_-, \quad
\lim_{N\to\infty} \frac{1}{N} \sum_{n=0}^{N-1} \rho_n \fq_n  =b_+.
\end{equation}
We denote $ D_\pm=\frac{b_\pm}{a_\pm}.$

Given $\mu_1, \mu_2$ let
$$ \cS_{\mu_1, \mu_2}(w)=\begin{cases} \frac{w}{\mu_1} & \text{if }w\geq 0\\ \frac{w}{\mu_2} & \text{if }w\leq 0.\end{cases}$$

Given $\theta, \gamma$ and $D$ we consider the following Markov process.
Let $\cW_D(u)$ be the Brownian motion with zero mean and variance $Du.$
Denote by $u^+(u)$ the total time on $[0,u]$ when $\cW$ is positive and
$u^-(u)$ the total time on $[0,u]$ when $\cW$ is negative. Given $t$ let $u_\gamma(t)$ be the solution of
$$ u^+(u_\gamma)+\frac{u^-(u_\gamma)}{\gamma}=t. $$
Set
$$ \cW_{\gamma, \theta, D}(t)=\cS_{\theta, 1} (\cW_D(u_\gamma(t))).$$
Note that this process is defined using the function $\cS$ with parameters $\theta$ and $1.$
Allowing more general parameters does not increase the generality since $\mu_2=1$
can always be achieved by rescaling because
$\cS_{\mu_1, \mu_2}(\cW_D(u_\gamma(\cdot)))$ has the same law as
$\cW_{\gamma, \frac{\mu_2}{\mu_1}, \frac{D}{\mu_1^2}}.$

In the case where
\begin{equation}
\label{GammaThetaSBM}
\gamma=\left(\frac{1-\fp}{\fp}\right)^2, \quad \theta=\frac{1-\fp}{\fp}
\end{equation}
the process $\cW_{\gamma, \theta, D}$ is referred to as the {\em skew Brownian Motion with parameter
$\fp$.} We will thus abbreviate $\cW_{\left(\frac{1-\fp}{\fp}\right)^2, \frac{1-\fp}{\fp}, D}$
as $\cB_{\fp, D}.$ Note that $\fp$ in \eqref{GammaThetaSBM} is given by
\begin{equation}
\label{ParSkew}
\fp=\frac{1}{\theta+1}.
\end{equation}

We refer the reader to \cite{Lej} for description of various equivalent definitions of the skew
Brownian Motion as well as its numerous applications. Of these definitions, the most relevant for us is
the following one \cite{HS}: $\cB_{\fp, 1}(t)$ is the scaling limit of
$ \frac{\fX(tN)}{\sqrt{N}}$ where $\fX$ is the random walk on $\integers$ which
moves to the left and to the right with probability $1/2$ everywhere
except the origin; at the origin $\fX$ moves to the right with probability $\fp$ and to the left with probability
$1-\fp.$

\subsection{Functional CLT}
\begin{theorem}
\label{ThSkewGen}
$W_N(t)=\frac{X_{Nt}}{\sqrt{N}}$ converges in law as $N\to\infty$
to $\cW_{\gamma, \theta, D}$ where
$$ \gamma=\frac{D_-}{D_+}, \quad
\theta=\frac{\mu_-}{\mu_+}, \quad D=\frac{D^+}{\mu_+^2}. $$
\end{theorem}
\begin{proof}
The proof of Theorem \ref{ThSkewGen} is very similar to the proof of Theorem \ref{ThCLTWalk} so we just sketch
the argument. As in Theorem \ref{ThCLTWalk} it suffices to show that
$ \hW_N$ defined by \eqref{DefhW}
converges to $\cW_{D^+} (u_\gamma(t)).$ We may assume without loss of generality that $D^-< D^+$
and so $\gamma<1.$ If this is not the case we could consider the reflected walk
$(-X(N), Y(N)).$
Let $\txi(u)=(\tX(u), \tY(u))$ be the following lazy walk. If $\tX\geq 0$ then its transition
probability coincides with $\mathfrak{P}.$ If $\tX<0$ then, with probability $1-\gamma,$
$\txi$ stays at its present location and with probability $\gamma$ it moves according to $\mathfrak{P}.$
There is a natural coupling between $\xi$ and $\txi$ such that
$\txi(u)=\xi(t(u)).$ Let $u(t)$ be the left inverse to $t(u).$ It is clear from the law of large numbers for
sums of geometric random variables that  with probability 1
$$ \lim_{t\to\infty} \frac{t}{u^+(u(t))+u^-(u(t))/\gamma}=1 $$
where $u^+(u)$ and $u^-(u)$ are occupation times of positive and negative semi-axis.
It therefore suffices to show that
\begin{equation}
\label{LazyFLT}
\tW_N\Rightarrow \cW_{D^+}
\end{equation}
where $\tW_N(u)=\frac{\fm(\txi(Nu))}{\sqrt{N}}$ and $\cW_{D^+}$ is the Brownian Motion with zero drift and
variance $D^+u.$
Note that
$\fm(\txi(u))$ is a martingale with quadratic variation $\sum_{n=1}^u \tfq(\txi(n))$
where
$$ \tfq(x,y)=\begin{cases} \fq(x,y) & \text{if } x\geq 0 \\
\gamma \fq(x,y) & \text{if } x<0. \end{cases} $$
According to \cite{Bn} it suffices to show that
\begin{equation}
\label{LimQVar}
\lim_{N\to\infty} \frac{\sum_{j=1}^N \tfq(\txi(j))}{N}=D^+.
\end{equation}
The proof of \eqref{LimQVar} is the same as the proof of Lemma \ref{LmEnvLLN}. The key step is to show
that if $L_N$ is the same as in the lemma, $|x|\leq K\sqrt{N}$ and $\ttau$
is the exit time from $[x-L_N, x+L_N]$ by $\txi$
then we have that for each $y\in\{1\dots m\}$
\begin{equation}
\label{LimQVTau} \tD_N(x,y)\approx D^+ \quad\text{where}\quad
\tD_N(x,y)=\frac{\EXP(\sum_{u=0}^{\ttau} \tfq(\txi(u))|\txi(0)=(x,y))}{\EXP(\ttau|\txi(0)=(x,y))}.
\end{equation}
To fix the ideas, suppose
that $[x-L_N, x+L_N]\subset (-\infty, 0)$ then
\begin{equation}
\label{QVNormalLazy}
\tD_N(x,y)=\gamma
\frac{\EXP(\sum_{u=0}^{\ttau} \fq(\txi(u))|\txi(0)=(x,y))}{\EXP(\ttau|\txi(0)=(x,y))}.
\end{equation}
Note that $\txi_N$ satisfies \eqref{MartEq}, \eqref{RhoEq} with
$$ \tfm(x,y)=\fm(x,y),\quad \trho(x,y)=\begin{cases} \rho(x,y) & \text{if } x\geq 0 \\
\frac{\rho(x,y)}{\gamma} & \text{if } x<0. \end{cases} $$
The computations in Section \ref{ScEnvLLN}, in particular, \eqref{AvEscape} and \eqref{GH},
applied to $\txi,$ show that the second factor in the RHS of \eqref{QVNormalLazy}
is asymptotic to $D_-$ so that
$$ \tD_N(x,y)\approx \gamma D^-=D^+$$
as claimed. Once \eqref{LimQVTau} is established the proof of \eqref{LimQVar}  proceeds as in Section
\ref{ScEnvLLN}.
\end{proof}

\subsection{Small perturbations of the environment.}
Consider an environment on $\bbS$ satisfying conditions
\eqref{XMNorm}, \eqref{AvOcc}, and \eqref{AvQV}. Let $\fP$ be defined by \eqref{striptransition}.
Consider a perturbation $\tilde\fP$ of $\fP$ such that
$$ \left|\bar\fP(z, z')-\fP(z, z')\right|\leq \frac{C}{|n|^\kappa+1}
\text{ where } z=(n,j) \text{ and } \kappa>1.$$

Let $\beta_n, \tbeta_n$ be sequences such that
\begin{equation}
\label{DefBetaTBeta}
\lambda_n=\frac{\beta_{n+1}}{\beta_n}, \quad
\tlambda_n=\frac{\tbeta_{n+1}}{\tbeta_n}.
\end{equation}
The following result is proven in Appendix \ref{AppReg}.

\begin{lemma}
\label{LmPertMatr}
(a) The following estimates hold
\begin{equation}
\label{AKey}
 ||\zeta_n-\brzeta_n||\leq \frac{C}{|n|^\kappa+1}, \quad
||l_n-\brl_n||\leq \frac{C}{|n|^\kappa+1}, \quad
||v_n-\brv_n||\leq \frac{C}{|n|^\kappa+1},
\end{equation}
\begin{equation}
\label{ACor}
||A_n-\brA_n||\leq \frac{C}{|n|^\kappa+1}, \quad
||\lambda_n-\brlambda_n||\leq \frac{C}{|n|^\kappa+1}, \quad
||\tlambda_n-\overline{\widetilde{\lambda}}_n||\leq \frac{C}{|n|^\kappa+1}.
\end{equation}
(b) The following limits exist
\begin{equation}
\label{LimBeta}
\beta_\pm=\lim_{n\to\pm\infty} \frac{\bar\beta_n}{\beta_n}=
\lim_{n\to\pm\infty} \frac{\overline{\tbeta}_n}{\tbeta_n}.
\end{equation}
(c) The perturbed walk satisfies
\eqref{XMNormPM}, \eqref{AvOccPM} and \eqref{AvQVPM} with
$$ \mu_\pm=\beta_\pm,   \quad
a_\pm=a/\beta_\pm,  \quad
b_\pm=b \beta_\pm
$$
where $a$ and $b$ are the limits of \eqref{AvOcc} and \eqref{AvQV} respectively
for the unperturbed walk.
\end{lemma}

For random walks on $\integers$ the lemma follows easily from the explicit expressions for
the objects involved. Namely
\begin{equation}
\label{MatZ}
 A_n=\frac{q_n}{p_n}=\lambda_n=\frac{\beta_{n+1}}{\beta_{n}}, \quad \text{and} \quad
\Delta_n=\beta_n,\quad
\rho_{n}=\frac{1}{\beta_n q_n}, \quad \fq_n=\beta_{n+1}\beta_n
\end{equation}
(see \cite[Section 5]{DG4}).
The case of the strip is more complicated and will be considered in Appendix \ref{AppReg}.

 Combining Theorem \ref{ThSkewGen} with Lemma \ref{LmPertMatr} we obtain the following result.
\begin{corollary}
\label{CrSkew}
Let $\brxi(t)=(\brX(t), \brY(t))$ denote the walk in the perturbed environment~$\overline{\fP}.$
$$ \frac{\brX(tN)}{\sqrt{N}}\Rightarrow \cB_{\fp, \; D} $$
where $D$ is the limiting variance of the walk in the unperturbed environment and
\begin{equation}
\fp=\frac{\beta_+}{\beta_++\beta_-}.
\end{equation}
\end{corollary}

\begin{remark}
Note that \eqref{DefBetaTBeta} does not define $\beta_n$ and $\tbeta_n$ uniquely. Namely, if we replace
$\beta_n$ by $c\beta_n$ and $\tbeta_n$ by $\tc\beta_n$ for any constants $c, \tc$ then
\eqref{DefBetaTBeta} remains valid. In this case $\beta_\pm$ get replaced by $\frac{\tc}{c} \beta_\pm$
but expression of $\fp$ does not depend on the arbitrariness involved in the choice of $c$ and $\tc.$
\end{remark}

For random walks on $\integers$ using the explicit expression for $\lambda_n$ in terms of $p_n$ and $q_n$ (see \eqref{MatZ})
we obtain
\begin{equation}
\label{UpsilonGen}
 \fp=\frac{\bupsilon}{\bupsilon+1}\quad\text{where}\quad
\bupsilon=\prod_{n=-\infty}^\infty \left(\frac{\tq_n p_n}{\tp_n q_n}\right) .
\end{equation}

\section{Semilocal Limit Theorem}
\label{ScSemiLoc}
We say that $X_N$ satisfies the semilocal limit theorem at the scale $L_N$ with \\$1\ll L_N\ll \sqrt{N}$ if
there exists a constant $\beta>0$ such that for each interval $I$ of length $L_N,$  for each $(x,y)$ with $|x|\leq N$
\begin{equation}
\label{SemiLocCLT}
\Prob(X_N-x\in I|\xi(0)=(x,y))=\Prob\left(\sqrt{DN} \cN\in I\right)+\cO\left(\frac{L_N^{1-\beta}}{\sqrt{N}}\right),
\end{equation}
where $\cN$ is the standard normal random variable and $D$ is a positive number (in our case $D$ comes from
Theorem \ref{ThCLTWalk}).

Clearly if for each $(x,y)$ with $|x|\leq N$ we have
$$ \sup_z \left|\Prob\left(\frac{X_N-x}{\sqrt{DN}} \leq z\Big|\xi(0)=(x,y)\right)-\Phi(z)\right|\leq N^{-\upsilon}
$$
then $X$ satisfies the semilocal limit theorem at the scale $N^\gamma$ for each
$\gamma>\frac{1}{2}-\upsilon.$
The next lemma allows us to decrease the scale in the semilocal limit theorem.

\begin{lemma}
\label{LmBoot}
Let $\eps, \eps_1<\eps_2$ be small positive constants.
If $N$ is sufficiently large and for each $\tN$ such that
$N^\eps \leq \tN\leq N,$
for each $(x,y)$ such that
$|x|\leq N(1+\eps_2),$  for each interval $I$ of length $L=\tN^\gamma$
where
$$\gamma<\frac{1}{2} \quad\text{and}\quad\gamma\left(\frac{1}{2}+\gamma\right)>\eps $$
we have
\begin{equation}
\label{BootTN}
\Prob(X_\tN-x\in I|\xi(0)=(x,y))=\mathbf{P}\left(\sqrt{D\tN} \cN\in I\right)+\cO\left(\frac{L^{1-\beta}}{\sqrt{\tN}}\right)
\end{equation}
then \eqref{SemiLocCLT} holds for all $(x,y)$ with $|x|<(1+\eps_1) N$ and $L_N=N^{(\frac{1}{2}+\gamma)\gamma}.$
\end{lemma}

Applying this lemma several times we obtain the following
\begin{corollary}
\label{CrBoot}
  Suppose that there exits $\gamma<\frac{1}{2}$ such that for each $\eps$ there are constants $\eps_1, \eps_2, N_0$
  such that the conditions
of Lemma \ref{LmBoot} are satisfied for $N\geq N_0.$ Then,
for arbitrarily small $\tgamma>0,$ $X$ satisfies the semilocal limit theorem at scale
$N^\tgamma.$
\end{corollary}

\begin{proof}[Proof of Lemma \ref{LmBoot}]
Throughout this proof we fix $(x,y)$ and let $\hat\Prob$ denote the distribution of $\xi$ under the condition that
$\xi(0)=(x,y).$

Let $s=\frac{1}{2}+\gamma.$ Note that $s<1.$
Consider  an interval $I$ of length $N^{\gamma s}.$ Let $N_1=N-N^s,$ $N_2=N^s.$
Divide $\integers$ into intervals $I_p$ of size $N^\gamma.$ Let $\brx$ be the center of $I$ and $x_p$ be the
centers of $I_p.$ Call $p$ {\em feasible} if
$$ |x_p-x|\leq N^{1/2+\eps} \quad \text{and} \quad |\brx-x_p|\leq N_2^{1/2+\eps}. $$
By the Azuma inequality, if $p$ is not feasible, then
$$ \hat\Prob(X_{N_1}\in I_p, \quad X_{N}\in I)\leq \exp(-N_2^{2\eps}). $$
Accordingly
\begin{equation}
\label{2Scales}
 \hat\Prob(X_N\in I)=\sum_{p-feasible} \hat\Prob(X_{N_1}\in I_p)\, \hat\Prob(X_N\in I|X_{N_1}\in I_p)+
\cO\left(e^{-N_2^{2\eps}}\right) .
\end{equation}
By \eqref{BootTN} each individual term in this sum is
$$\frac{N^\gamma}{\sqrt{2\pi DN_1}} e^{-(x_p-x)^2/(2D N_1)}
\times
\frac{N^{\gamma s}}{\sqrt{2\pi DN_2}} e^{-(\brx-x_p)^2/(2D N_2)}
+\cO\left(N^{(\gamma-\frac{1}{2})(1+s)-\beta s}\right). $$
Since $p$ is feasible we can replace
$$ e^{-(x_p-x)^2/(2D N_1)} \quad\text{by} \quad e^{-(\brx-x)^2/(2D N_1)}  $$
with an error of order $\cO(N^{-\eps}).$
Accordingly the main contribution to \eqref{2Scales} comes from
$$\frac{N^{\gamma s}}{\sqrt{2\pi DN_1}} e^{-(x-\brx)^2/(2D N_1)}
\sum_p
\frac{N^{\gamma }}{\sqrt{2\pi DN_2}} e^{-(\brx-x_p)^2/(2D N_2)}$$
$$=
\frac{N^{\gamma s}}{\sqrt{2\pi DN_1}} e^{-(\brx-x)^2/(2D N_1)}
\left[\int_{-\infty}^{\infty} \frac{1}{\sqrt{2\pi D}} e^{-(\brx-z)^2/(2D)}\; dz+\cO\left(\frac{N^\gamma}{\sqrt{N_2}}
  \right)\right]$$
$$=
\frac{N^{\gamma s}}{\sqrt{2\pi DN_1}} e^{-(\brx-x)^2/(2D N_1)}
\left[1+\cO\left(N^{\frac{\gamma-1/2}{2}}\right)\right]
$$
where the first equality is obtained by replacing the Riemann sum with step $\hbar=\frac{N^\gamma}{\sqrt{N_2}}$
with the Riemann integral with mistake
$\cO(\hbar).$ The result follows.
\end{proof}

\section{Environment viewed by the particle: Mixing.}
\label{ScEnvMix}
\subsection{General result.}
Here we provide sufficient conditions for mixing of the environment viewed by the particle process.
Namely we assume that  there is a sequence $\delta_N$
converging to $0$ as $N\to\infty$, such that for each $\eps, K$ there exists $N_0$ such that for $N\geq N_0$
for each $k$ with $|k|\leq K \sqrt{N}$
\begin{equation}
\label{UnifConv2}
\left|\frac{1}{2 \delta_N N^{1/4}}\sum_{j=k-\delta_N N^{1/4}}^{k+\delta_N N^{1/4}}\rho_j\one
-a\right|\leq \eps.
\end{equation}
We consider functions $h:\bbS\to\reals$ satisfying \eqref{UnifConv1}.

\begin{theorem}
\label{ThEnvMix}
If \eqref{UnifConv1},
\eqref{CanStay} and \eqref{UnifConv2} hold then
$\EXP(h(\xi(N)))\to \frac{\fh}{a}$ as $N\to\infty.$
\end{theorem}
\begin{proof}
If \eqref{UnifConv1}, and \eqref{UnifConv2}
hold then the argument of Section \ref{ScEnvLLN} shows that for each $\eps, \delta, K$
there exists $N_1$ such that for $N\geq N_1$
and for each $(k,y)\in\bbS$ such that $|k|\leq K\sqrt{N}$  we have
\begin{equation}
\label{UnifEnvLLN}
\Prob\left( \left|\frac{1}{\delta\sqrt{N}}\sum_{j=0}^{\delta \sqrt{N}-1} h(\xi(j))-\frac{\fh}{a}\right|>\eps
\;\Big|\;\xi(0)=(k,y)\right)<\frac{\eps}{||h||_\infty}.
\end{equation}
That is, the conclusion of Lemma \ref{LmEnvLLN} holds uniformly for initial conditions $(k,y)$
satisfying $|k|<K\sqrt{N}.$

Given a trajectory $\xi$ we denote by $\txi$ the accelerated trajectory which skips all steps where
$\xi$ stays at the same place. That is, $\txi(n)=\xi(t(n))$ where
$t(0)=0$ and for $n>0,$ $t(n)=\min(t>t(n-1): \xi(t)\neq \xi(t(n-1))).$
We denote by $s(n)=t(n+1)-t(n)$ the time the walker spends at $\txi(n).$

A {\em path} is a finite set of points $W=\{z_0, z_1\dots z_{l}\}$ such that $z_j\neq z_{j+1}$ for $j=0, \dots, l-1.$
The number $l=l(W)$ is called the {\em length of the path}.
A path is called {\em admissible} if there is an accelerated trajectory
$\txi$ such that $\txi(n)=z_n$ for $0\leq n\leq l(W).$
Given an admissible path $W$ and a trajectory $\xi$ following this path let $\tau(W,\xi)=t(l(W))$ be the number of steps
it takes $\xi$ to traverse this path.
Let $T(W)$ be the expectation of $\tau(W, \xi)$
conditioned on the event that $W$ is the beginning part of $\txi.$
Observe that
\begin{equation}
\label{SumGeom}
\tau(W, \xi)=\sum_{n=0}^{l-1} s(n)
\end{equation}
 where, for a fixed $W,$ $s(n)$ are independent random variables
having geometric distributions with parameter
$1-\mathfrak{P}(\txi(n), \txi(n)).$

Let $\cS(N)$ be the
set of (admissible) paths such that $T(W)\geq \frac{N}{2}$ but $T(W^-)<\frac{N}{2}$ where $W^-$ is the path obtained
by removing the last edge from $W$.
Given $W\in \cS(N),$ $\delta, j$ let
$$ A_{W,\delta, j}=\left\{\xi: W=\txi([0, l(W)])\text{ and } \tau(W, \xi)\in
\left[\frac{N}{2}+\delta j \sqrt{N}, \frac{N}{2}+\delta (j+1) \sqrt{N}\right)\right\}.$$
By the Central Limit Theorem for $\tau(W, \xi),$ (see \eqref{SumGeom}),
given $\eps>0$ we can find $R$ such that
$$\Prob\left(\bigcup_{W\in \cS(N)} \left\{\xi: \left|\tau(W, \xi)-\frac{N}{2}\right|>R\sqrt{N}\right\}\right)
 \leq \eps. $$
Accordingly
$$\EXP(h(\xi(N)))=\left[\sum_{W\in \cS(N)} \sum_{|j|\leq R}
\Prob(A_{W, \delta, j})\;\EXP\Big(h(\xi(N))\big|A_{W,\delta, j}\Big)\right]+\teps $$
where $\teps\leq (\sup |h|)\eps.$

We claim that
\begin{equation}
\label{ConstEps}
 \left|\sum_{W\in \cS(N)} \sum_{|j|\leq R/\delta}
\Prob(A_{W, \delta, j})\; \EXP\Big(h(\xi(N))|A_{W,\delta, j}\Big)-\frac{\fh}{a}\right|\leq 3\eps
\end{equation}
provided that $\delta$ is small enough.
Indeed
\begin{equation}
\label{CondDelay}
 \EXP\Big(h(\xi(N))|A_{W,\delta, j}\Big)=
 \sum_{l=1}^{\delta  \sqrt{N}}
\EXP \left(h\left(\xi\left(\frac{N}{2}-\delta (j+1) \sqrt{N}+l\right)\right)\big|\,\xi(0)=e(W)\right)
\end{equation}
$$\quad\quad\quad\quad\quad\quad\quad\quad\quad\quad
\times \Prob\left(\tau(W,\xi)=\frac{N}{2}+\delta (j+1)\sqrt{N}-l\,\big|\,A_{W,\delta, j}\right)
$$
where $e(W)=(x(W), y(W))$ is the endpoint of $W.$
By the Local Limit Theorem for the sum \eqref{SumGeom} (\cite{Pr, DMD})
\begin{equation}
\label{KeySpaceTme}
\left|\Prob\left(\tau(W,\xi)=\frac{N}{2}+\delta (j+1)\sqrt{N}-l\,\big|\,A_{W,\delta, j}\right)-\frac{1}{\delta\sqrt{N}}
\right|\leq \frac{\eps}{||h||_\infty}
\end{equation}
uniformly in $l\leq \delta\sqrt{N}$ provided that $\delta$ is small enough.
This allows us to replace $\EXP\Big(h(\xi(N))|A_{W,\delta, j}\Big)$
by
$$  \frac{1}{\delta\sqrt{N}} \sum_{l=1}^{\delta \sqrt{N}}
\EXP\left(h\left(\xi\left(b_{N,j,l}\right)\right)\,\big|\,\xi(0)=e(W)\right), $$
where
\begin{equation}
\label{DefBNJL}
  b_{N,j,l}=\frac{N}{2}-\delta (j+1) \sqrt{N}+l.
\end{equation}

To control this sum we 
consider two cases.

(I) The terms where $|x(W)|$ is large can be controlled as follows.
By Theorem \ref{ThCLTWalk}
$$  \frac{1}{\delta\sqrt{N}} \sum_{W\in \cS(N)} \sum_{|j|\leq R/\delta} \sum_{|x(W)|>K\sqrt{N}}
\sum_{l=1}^{\delta j \sqrt{N}}
\Prob(A_{W, \delta, j}) \Prob\left(\xi\left(b_{N,j,l}\right)=e(W)\Big|\,
\xi(0)=e(W)\right)
 $$
$$\leq \Prob\left(|\xi(N)|>\frac{K\sqrt{N}}{2}\right)\leq \frac{\eps}{||h||_\infty} $$
provided that $K$ is sufficiently large and $N\geq N_2(K).$

(II) On the other hand if $|x(W)|\leq K \sqrt{N}$ then in view of \eqref{UnifEnvLLN}
$$  \left|\frac{1}{\delta \sqrt{N}} \sum_{l=1}^{\delta \sqrt{N}}
\EXP\left(h\left(\xi\left(b_{N,j,l}\right)\right)
\Big|\,\xi\left(0\right)=e(W)\right)-
\frac{\fh}{a}\right|\leq \eps $$
provided that $N$ is large enough.

Combining the estimates for the cases (I) and (II) above with
\eqref{KeySpaceTme}  we obtain  \eqref{ConstEps}.
Since $\eps$ is arbitrary Theorem \ref{ThEnvMix} follows.
\end{proof}

\subsection{Examples.}
Examples presented in Section \ref{ScEx} also satisfy \eqref{UnifConv1} and \eqref{UnifConv2}.

In fact, in Example \ref{ExQP} we can replace quasiperiodic environments by more
general environments generated by uniquely ergodic transformation (we refer the reader to \cite[\S 1.8]{CFS}
(for background on uniquely ergodic transformations).
That is, let $T$ be a uniquely ergodic transformation
of a compact metric space $\Omega,$ $(P, Q, R)_n(\omega)=(\cP, \cQ, \cR)(T^n \omega)$
and $h_n(\omega)=\cH(T^n \omega).$ 

\begin{proposition}
\label{PrUC-QP}
If $(\cP, \cQ, \cR)$ and $\cH$ are continuous then \eqref{UnifConv1} and \eqref{UnifConv2} hold.
\end{proposition}

\begin{proof}
By Section 6 and Appendix A of \cite{DG4}, $\rho_n=\brho(T^n \omega),$ where
$\brho:X\to\reals^m$ is continuous. Therefore  \eqref{UnifConv1} and \eqref{UnifConv2} follow from the
fact that the convergence in ergodic theorem for uniquely ergodic systems is uniform with respect to $\omega$
(\cite[Theorem 1.8.2]{CFS}).
\end{proof}

In the case of  independent environments we suppose that  $h_n=\cH(P_n, Q_n, R_n)$ where $\cH$ is a bounded
continuous function.

\begin{proposition}
\label{PrUC-Ind}
\eqref{UnifConv1} and \eqref{UnifConv2} hold for independent environments.
\end{proposition}
\begin{proof}
The proof of \eqref{UnifConv2} is very similar to the proof of Proposition \ref{PrMIM-Ind} so it can be left to the reader.
The proof of \eqref{UnifConv1} in case $\cH$ is a local function
(that is there exists $R$ such that $\cH(\omega)$ depends only on $(P_n, Q_n, R_n)_n$ with $|n|\leq R$)
is also similar to Proposition \ref{PrMIM-Ind}. To prove \eqref{UnifConv1} for general continuos function, it suffices
to approximate it by a local function with error less than $\eps/2.$
\end{proof}

Propositions \ref{PrUC-QP} and \ref{PrUC-Ind} complete the proof of Theorem \ref{ThEnvMixEx} for Examples \ref{ExQP}
and~\ref{ExInd}. To prove Theorem \ref{ThEnvMixEx} for Example \ref{ExPert} we need to take into account that
the walk is not allowed to remain at the same site at two consecutive moments of time.
Because of that, we consider $\xi$ at odd and at even
times separately and note that \eqref{UnifConv-Per} implies \eqref{UnifConv1} for both odd and even sublattices.

\section{Local Limit Theorem}
\label{ScLLT}
\begin{theorem}
\label{ThLLT}
If \eqref{RCMart}, \eqref{RC-Occ-V},  and \eqref{RC-Occ-Q}
hold
then
for each sequence $(k_N, y_N)$ such that $k_N/\sqrt{N}$ is bounded
we have
\begin{equation}
\label{LLT-a}
\lim_{N\to\infty}  \frac{\mathbb{P}\left(\xi(N)=(k_N, y_N)\right)}{\bP\left(\sqrt{\frac{bN}{a}}\cN\in
\left[k_N-\frac{1}{2}, k_N+\frac{1}{2}\right]\right)\rho(k_N,y_N)}=\frac{1}{a}
\end{equation}
where $a$ and $b$ are the constants from \eqref{RC-Occ-V} and \eqref{RC-Occ-Q} respectively.
\end{theorem}

\begin{proof}
  We use the same notation as in Section \ref{ScEnvMix}. In particular we choose a small constant
  $\eps_2$ and let $\delta$ be as in the proof of Theorem \ref{ThEnvMix}.
  We have
\begin{equation}
\label{Deco}
 \mathbb{P}(\xi(N)=(k_N,y_N))=\sum_W \sum_j
\mathbb{P}(A_{W, \delta, j}) \mathbb{P}\Big(\xi(N)=(k_N,y_N)|A_{W, \delta, j}\Big),
\end{equation}
where the sum is over all admissible paths $W.$

Given $\brR$ denote by
$\cS_\brR(N)$ the set of paths in $\cS(N)$ whose endpoint $e(W)=(x(W), y(W))$ satisfies $|x(W)|\leq \brR\sqrt{N}.$
Pick $\brR\gg 1$ and divide the sum \eqref{Deco} into three parts.

(I) If $W\in \cS_\brR(N)$ and $|j|<N^\breps $ then \eqref{KeySpaceTme} allows us to replace
$$ \mathbb{P}(\xi(N)=(k_N,y_N)|A_{W, \delta, j}) $$
by
$$ \frac{1}{\delta\sqrt{N}} \sum_{l=1}^{\delta \sqrt{N}} \mathbb{P}\Big(\xi(b_{N,j,l})=(k_N, y_N)|\xi(0)=e(W)\Big).
$$
where $b_{N,j,0}=\frac{N}{2}-\delta (j+1) \sqrt{N},$ $b_{N, j,l}=b_{N, j, 0}+l$ (see \eqref{DefBNJL}).
Divide
$\integers$ into segments $I_p$ of length $L_N=N^{1/5}.$ Let $\tk_p$ be the center of $I_p.$
We split the above sum as
$$
\sum_p \sum_{\tk \in I_p \atop \ty\in \{1,\dots,m\}}
\mathbb{P}(\xi(b_{N,j,0})=(\tk, \ty)|\xi(0)=e(W))
\;\;\frac{\sum_{l=1}^{\delta \sqrt{N}} \mathbb{P}(\xi(l)=(k_N, y_N)|\xi(0)=(\tk, \ty))}{\delta\sqrt{N}}
$$

Denote  $\brN=\delta^{1/2} N^{1/4}.$ By Corollary \ref{CrExpVis}
if $|k_N-\tk_p|\leq \brR\sqrt{\brN}$ and $\tk\in I_p$ then
$$ \frac{\sum_{l=1}^{\delta \sqrt{N}} \mathbb{P}(\xi(l)=(k_N, y_N)|\xi(0)=(\tk, \ty))}{\delta\sqrt{N}}\sim
\frac{\bE(\fl_{(k_N-\tk)/\brN), 1/a})}{\brN} \rho_{k_N, y_N} $$
\begin{equation}
\label{AvLLT}
\hskip6.1cm
\sim \frac{\bE(\fl_{(k_N-\tk_p)/\brN), 1/a})}{\brN} \rho_{k_N, y_N}
\end{equation}
where the last step uses that $|k_N-\tk_p|\ll \sqrt{N}. $

On the other hand Corollary \ref{CrBoot}
and Theorem \ref{ThCLTRate} show that
$$ \sum_{\tk \in I_p \atop \ty\in \{1,\dots,m\}} \mathbb{P}(\xi(b_{N,j,0})=(\tk, \ty)|\xi(0)=e(W))
=\frac{L_N(1+o_{N\to\infty}(1))}{\sqrt{\pi DN}}
\exp\left(-\frac{(x(W)-\tk_p)^2}{DN}\right)
$$
\begin{equation}
\label{SegSum}
=\frac{L_N(1+o_{N\to\infty}(1))}{\sqrt{\pi DN}}
\exp\left(-\frac{(x(W)-k_N)^2}{DN}\right)
\end{equation}
where $DN=2D(\frac{N}{2})$ appears in the above expression since
$b_{N,j,0}=\frac{N}{2}+\cO\left(N^{(1/2)+\eps}\right).$

Next, if
\begin{equation}
\label{IntRange}
|k_N-\tk_p|\geq \brR\sqrt{\brN}
\end{equation}
then
\begin{equation}
\label{DeltaNSum}
\frac{\sum_{l=1}^{\delta \sqrt{N}} \mathbb{P}(\xi(l)=(k_N, y_N)|\xi(0)=(\tk, \ty))}{\delta\sqrt{N}}\leq
\end{equation}
$$\mathbb{P}\left(\xi\text{ visits } (k_N, y_N) \text{ before time } \delta\sqrt{N}|\xi(0)=(\tk, \ty)\right)\times$$
$$\mathbb{E}(\Card(l\leq \sqrt{N}: \xi(l)=(k_N, y_N))|\xi(0)=(k_N, y_N)) .$$
The first factor is $\cO\left(e^{-c(k_N-\tk)^2/[\brN^2]} \right)$
by the Azuma inequality
and the second factor is $\cO(\brN)$ by  Lemma \ref{LmUILT},
so in case \eqref{IntRange} we have
\begin{equation}
\label{OccFar}
\frac{\sum_{l=1}^{\delta \sqrt{N}} \mathbb{P}(\xi(l)=(k_N, y_N)|\xi(0)=(\tk, \ty))}{\delta\sqrt{N}}
\leq \frac{C}{\brN} \exp\left(-\frac{c(k_N-\tk_p)^2}{\brN^2}\right).
\end{equation}
Hence (see \eqref{SegSum})
$$ \sum_{\tk \in I_p \atop \ty\in \{1,\dots,m\}} \mathbb{P}(\xi(b_{N,j,0})=(\tk, \ty)|\xi(0)=e(W))\;\;
\frac{\sum_{l=1}^{\delta \sqrt{N}} \mathbb{P}(\xi(l)=(k_N, y_N)|\xi(0)=(\tk, \ty))}{\delta\sqrt{N}} $$
\begin{equation}
\label{FarSum}
\leq \frac{C L_N}{\brN \sqrt{N}} \exp\left(-\frac{c(k_N-\tk_p)^2}{\brN^2}\right).
\end{equation}

Next, we perform the summation over $p.$
Equations \eqref{AvLLT}, \eqref{SegSum}, \eqref{FarSum} show that in case (I)
$$ \frac{1}{\delta\sqrt{N}} \sum_{l=1}^{\delta \sqrt{N}} \mathbb{P}(\xi(b_{N,j,l})=(k_N, y_N)|\xi(0)=e(W))$$
$$ =\sum_{|k_p-\tk_N|<\brR\sqrt{\brN}}
\frac{L_N(1+o_{N\to\infty, \brR\to\infty }(1))}{\sqrt{\pi DN}}
\frac{\bE(\fl_{(k_N-\tk_p)/\brN), 1/a})}{\brN} \rho_{k_N, y_N}
\exp\left(-\frac{(x(W)-k_N)^2}{DN}\right) $$
\begin{equation}
\label{NOver2}
=\frac{1}{\sqrt{\pi D N}}  \;\frac{\rho_{k_N, y_N}}{a}
\exp\left[-\frac{(x(W)-k_N)^2}{DN}\right]\;
\left(1+o_{N\to\infty, \brR\to\infty}(1)\right)
\end{equation}
where the last step relies on the fact that
$$ \int_{-\infty}^\infty \fl_{x, t} dx=t .$$


(II) $W\not\in \cS_\brR(N)$ and $|j|<N^\breps .$ In this case the same argument as in the proof
of \eqref{OccFar} shows that
$$ \frac{1}{\delta\sqrt{N}} \sum_{l=1}^{\delta \sqrt{N}}
\mathbb{P}\Big(\xi(b_{N,j,l})=(k_N, y_N)|\xi(0)=e(W)\Big)\leq \frac{\eps(\brR)}{\sqrt{N}}, $$
where $\eps(\brR)\to 0$ as $\brR\to\infty.$

(III) $|j|\geq N^\breps.$ 
Due to moderate deviation
estimate for sums of independent random variables applied to the sum~\eqref{SumGeom}.
$$ \Prob\left(\bigcup_W \bigcup_{j\geq N^\breps} A_{W, \delta, j} \right)\leq C e^{-cN^{2\breps}}. $$
Thus the main contribution to \eqref{Deco} comes from case (I). Performing the summation over
$W\in \cS_\brR(N)$ and $j\in [-N^\breps, N^\breps]$ and using \eqref{NOver2} and the CLT for $x(W)$
we obtain \eqref{LLT-a}.
\end{proof}

Theorem \ref{ThLLT} implies Theorem \ref{ThLLTEx}(a). To prove Theorem \ref{ThLLTEx}(b) we need to consider
$\xi(2N)$ and $\xi(2N+1)$ separately (see the discussion at the end of Section \ref{ScEnvMix}) and note that in Example
\ref{ExPert} $D=1$ since, due to equation \eqref{RM-SRW}, $\xi$
is a small perturbation of the simple random walk away from the origin.

\appendix
\section{A rough bound on large and moderate deviations.}
\label{AppLargeMod}

\begin{proposition}
  \label{PrLLNMart}
  Let $\{\cF_n\}$, $n\ge0,$ be a filtration and
$B_n$ be a sequence of $\cF_n$-measurable random variables such that
$B_0=0$ and $\Delta_n=B_n-B_{n-1}$ satisfies
for $n\leq N$ the following estimates:
\begin{equation}
\label{SubMart}
\bE(\Delta_n|\cF_{n-1})\leq -\eps \quad\text{where}\quad \eps\geq N^{-1/2}
\end{equation}
and for some positive constants $c, K$
\begin{equation}
\label{Cramer}
\bE(e^{c |\Delta_n|}|\cF_{n-1})\leq K.
\end{equation}

Then there is a constant $\brc=\brc(c, K)>0$ such that
$$\bP(B_N\geq 0)\leq \begin{cases} e^{-\brc \sqrt{\eps N}} & \textrm{if }\eps\geq N^{-1/3} \\
 N e^{-\brc \eps^2 N}  & \textrm{if }\eps<N^{-1/3}. \end{cases} $$
\end{proposition}

\begin{remark}
The first case ($\eps\geq N^{-1/3}$) is sufficient for all the applications given in this paper except that
one would get worse constants in Section \ref{SSRates}.
\end{remark}

\begin{proof}
Suppose first that $\eps\geq N^{-1/3}.$
  Let $s=\frac{c_1}{\sqrt{\eps N}}$ for a sufficiently small constant $c_1$
  (see \eqref{ExpQuad} and \eqref{CharLessOne} below for the precise conditions on $c_1.$) Set $A=\sqrt{\eps N}$
  and define
  $$ \tDelta_k=\Delta_k 1_{\Delta_k<A}, \quad
\tB_n=\sum_{k=1}^{n}\tDelta_k, \quad \phi_n=\bE\left(e^{s \tB_n}\right). $$
Then
$$ \phi_k=\bE\left(e^{s\tB_{k-1}} \bE\left(e^{s\tDelta_k}|\cF_{k-1}\right)\right). $$
Note that $s \tDelta_k\leq s A= c_1$ and so we can choose $c_1$ so small that
\begin{equation}
  \label{ExpQuad}
  e^{s\tDelta_k}\leq 1+s\tDelta_k+(s\tDelta_k)^2,
\end{equation}
and so
\begin{equation}
  \label{ExpQuadCond}
  \bE\left(e^{s\tDelta_k}|\cF_{k-1}\right)\leq 1+s\bE\left(\tDelta_k|\cF_{k-1}\right)+s^2 \bE\left(\tDelta_k^2|\cF_{k-1}\right).
\end{equation}
In view of \eqref{Cramer}
\begin{equation}
\label{M2Cond}
\bE(\tDelta^2|\cF_{k-1})\leq \bE(\Delta^2|\cF_{k-1})\leq \Const
\end{equation}
and since $A=\sqrt{\eps N}\geq N^{2/3}$ we have that for large $N$
\begin{equation}
\label{M1Cond}
\EXP(\tDelta_k|\cF_{k-1})\leq -\frac{2\eps}{3}.
\end{equation}
Note that $\DS \frac{s^2}{\eps s}=\frac{c_1}{\sqrt{\eps^3 N}}\leq c_1$ can be made as small as we wish by choosing $c_1$ small.
Hence \eqref{ExpQuadCond}, \eqref{M2Cond} and \eqref{M1Cond} show that we can choose $c_1$ so small that
\begin{equation}
\label{CharLessOne}
\bE\left(e^{s\tDelta_k}|\cF_{k-1}\right)\leq
1-\frac{\eps s}{2}.
\end{equation}
Accordingly
$$ \bE\left(e^{s \tB_N}\right)\leq \left(1-\frac{\eps s}{2}\right)^N . $$
Thus for large $N$
\begin{equation}
\label{ControlTB}
\bP(\tB_N\geq 0)\leq e^{-s\eps N/4}.
\end{equation}
Next for each $n$
$$ \Prob(\Delta_n\geq A)\leq \Prob\left(e^{c \Delta_n}\geq e^{cA}\right)\leq   K e^{-cA}.$$
Hence
\begin{equation}
\label{AppAbove}
\bP(\tB_N\neq B_N)\leq N \max_{n\leq N} \Prob(\Delta_n\geq A)\leq
N e^{-c A}
\end{equation}
where the last inequality follows by \eqref{Cramer}.
Combining \eqref{ControlTB} with \eqref{AppAbove}
and using that $\eps s N=c_1\sqrt{\eps N},$ $A=\sqrt{\eps N}$
we obtain the required estimate in case $\eps\leq N^{-1/3}.$

Next consider the case where $\eps<N^{1/3}.$ We can also assume that
$\DS \eps \geq \sqrt{\frac{\ln \ln N }{N}}$
since the result is trivial (and useless) if $\eps^2 N<\ln \ln N$ because the RHS is greater than 1.
The argument in the case where
$\DS \sqrt{\frac{\ln \ln N }{N}}\leq \eps < N^{-1/3} $
is the same as in the case where $\eps>N^{-1/3}$
except that the parameters are chosen differently. Namely, we let $s=c_1 \eps$ where $c_1$ is
appropriately small and $A=N \eps^2.$ With this choice of parameters both
$sA=c_1 N \eps^3$ and $\frac{s}{\eps}=c_1$ still can be made as
small as needed. Accordingly we still have
$$ \bP(B_N\geq 0)\leq e^{-s\eps N/4}+N e^{-cA} $$
giving the required bound.
\end{proof}

\begin{remark}
We will often use the following consequence of Proposition \ref{PrLLNMart}: for any $\delta_1$ there are
positive constants
$C_1, C_2$ and $\delta_2$
such that if \eqref{SubMart} and \eqref{Cramer} hold and
\begin{equation}
\label{LargeDrift}
\eps>N^{\delta_1-\frac{1}{2}}
\end{equation}
then
\begin{equation}
\label{SEBound}
\bP(B_N\geq 0)\leq C_1 e^{-C_2 N^{\delta_2}}.
\end{equation}
\end{remark}

\section{Contraction properties of products of positive matrices.}
\label{AppPM}
The proof of relations \eqref{v_n}, \eqref{v_n1} follows from very general and well known contracting
properties of positive matrices which we now recall.

Let $\mathbb{A}_{\delta}$ be the set of positive $m\times m$ matrices such that for any $A=(A(i,j))\in\mathbb{A}_{\delta}$
one has $\min_{i,j,k} A(i,k)/A(j,k)\ge\delta$, where $\delta>0$ does not depend on $A$.
Let $\mathbb{R}_+^m$ be the cone of non-negative vectors in $\mathbb{R}^m$ and  $\mathbb{R}_{+,\delta}^m$ its sub-cone
of positive column vectors  with $\min_{i,j} x_i/x_j\ge\delta$.
Then $A\mathbb{R}_{+}^m\subset \mathbb{R}_{+,\delta}^m$ for any $A\in \mathbb{A}_{\delta}$.
Indeed, for any vector $x\ge0\ (x\not=0)$ we have
\begin{equation}\label{EqContr2}
\frac{(Ax)_i}{(Ax)_j}=\frac{\sum_{k=1}^m A(i,k)x_k}{\sum_{k=1}^m A(j,k)x_k}\ge \min_{k} \frac{A(i,k)}{A(j,k)}\ge \delta.
\end{equation}
Next denote by $\mathcal{C}^\delta$ the set of rays generated by vectors from $\mathbb{R}_{+,\delta}^m$.
Also, we introduce the convention that $\mathcal{C}^0$ is the set of rays generated by vectors from $\mathbb{R}_{+}^m$.
If $\bx, \by\in \mathcal{C}^\delta$ are two rays generated by vectors $x, y\in \mathbb{R}_{+,\delta}^m$ then
the Hilbert's projective distance between them is defined by
\[
\mathfrak{r}(\bx,\by)=\max_{i,j}\ln \frac{x_iy_j}{x_jy_i}.
\]
The set $\mathcal{C}^\delta$ equipped with this metric is a compact metric space. The action of
a matrix $A\in\mathbb{A}_{\delta}$ on $\mathcal{C}^0$ is naturally defined by its action on $\mathbb{R}_{+}^m$
and for $\bx\in \mathcal{C}^0$ we write $A\bx$ for the image of $\bx$ under the action of $A$.
\eqref{EqContr2} shows that in fact $A\mathcal{C}^0\subset \mathcal{C}^\delta$.

We need the following version of a (stronger) result from \cite[Chapter XVI, Theorem 3]{Bf}:
for all $A\in\mathbb{A}_{\delta}$ and all $\bx,\by\in \mathcal{C}^\delta$
\begin{equation}\label{EqContr}
\mathfrak{r}(A\bx,A\by)\le c\, \mathfrak{r}(\bx,\by),\ \text{ where }\ c=\frac{1-\delta}{1+\delta}.
\end{equation}

We are now in a position to prove \eqref{v_n} and \eqref{v_n1} from  Section \ref{ScMatrices}.
To this end, note first that \eqref{EqPositive} implies that $A_n\in \mathbb{A}_{\delta}$
with $\delta=m\bar{\varepsilon}^2$.

Next, for $a\le n$, the sets $\mathcal{C}_a\de A_n...A_{a}\mathcal{C}^0$ form a decreasing sequence,
$\mathcal{C}_a\supset \mathcal{C}_{a-1}$,
of compact subsets of $\mathcal{C}^\delta$ and therefore $\bigcap_{a\le n}\mathcal{C}_a \not=\emptyset.$
 Due to \eqref{EqContr}, for any two rays $\bx,\by\in\mathcal{C}^\delta$
the projective distance between their images in $\mathcal{C}_a$ decays exponentially as $a\to-\infty$:
\begin{equation}\label{EqContr1}
\mathfrak{r}(A_n...A_{a}\bx,A_n...A_{a}\by)\le c^{n-a} \mathfrak{r}(\bx,\by).
\end{equation}
(There is no loss of generality in assuming that $\bx,\by\in\mathcal{C}^\delta$
since $A_{a}\mathcal{C}^0\subset \mathcal{C}^\delta$.)

Therefore there is a unique ray $\bv_n=\bigcap_{a\le n}\mathcal{C}_a $ and $v_n$ in \eqref{v_n} is the unit vector
corresponding to $\bv_n$ which proves \eqref{v_n}. It remains to note that at a small scale the standard
distance between unit vectors (as in \eqref{v_n1}) is equivalent to the distance between rays generated by these vectors
which means that \eqref{EqContr1} is equivalent to \eqref{v_n1}.

\section{Regularity of $\rho$ and $\Delta.$}
\label{AppReg}
Here we discuss the regularity of $\rho$ and $\Delta$ which plays a key role in our analysis.
To this end we recall the formulas for these expressions obtained in \cite{DG4}.

Let $\Omega$ be a compact metric space and $T:\Omega\to \Omega$ be a
continuous map. (This meaning for the letter $T$ is reserved for Appendix \ref{AppReg} only.)

Throughout this section we assume that $(P, Q, R)_n(\omega)=(\cP, \cQ, \cR)(T^n \omega)$ where
$(\cP, \cQ, \cR)$ are continuous functions such that \eqref{stoch} and \eqref{EqC2*} are satisfied. Define
\begin{equation}
\label{ZetaAV}
\begin{aligned}
&\zeta(\omega)=\zeta_0(\omega), \quad A(\omega)=A_0(\omega), \quad \alpha(\omega)=\alpha_0(\omega), \quad \sigma(\omega)=\sigma_0(\omega)\\
&v(\omega)=v_0(\omega),\quad  l(\omega)=l_0(\omega)
\quad \lambda(\omega)=\lambda_0(\omega), \quad \tlambda(\omega)=\tlambda_0(\omega)
\end{aligned}
\end{equation}
then
\begin{equation}\label{defLambda}
\begin{aligned}
& \zeta_n=\zeta(T^n\omega),\quad A_n=A(T^n \omega),\quad \alpha_n=\alpha(T^n\omega), \quad \sigma_n(\omega)=\sigma(T^n\omega),\\
& v_n=v(T^n\omega),\quad l_n=l(T^n\omega),\quad \lambda_n=\lambda(T^n \omega),\quad \tlambda_n=\tlambda(T^n \omega).
\end{aligned}
\end{equation}
It is proven in \cite{DG4} that RWRE in bounded potential enjoy the property that
\begin{equation}
\label{CoBoundary}
\tlambda(\omega)=\frac{\tbeta(T\omega)}{\tbeta(\omega)} , \quad
\lambda(\omega)=\frac{\beta(T\omega)}{\beta(\omega)}
\end{equation}
for continuous functions $\beta, \tbeta.$
Moreover, the functions $\zeta(\cdot)$, $v(\cdot)$, $l(\cdot)$
are continuous in $\omega$.
The continuity of all other functions is implied by the continuity of $\zeta$, $v$, and $l$.

It is proven in \cite{DG4} that
$$ \rho_n(\omega)=\brho(T^n \omega) \text{ and }\Delta_n(\omega)=\bDelta(T^n\omega) $$
where
\begin{equation}
\label{RhoDelta}
\brho(\omega)=c \frac{l(\omega)}{\tbeta(\omega)}\quad \text{and} \quad
\bDelta(\omega)=\beta(T \omega) \sigma(\omega) v(\omega)+\cB(T\omega)-\cB(\omega)
\end{equation}
and
\begin{equation}
\label{DefcB}
 \cB(\omega)=\sum_{k=0}^{\infty}\beta(T^{k+1}\omega)\left[\zeta_0\dots\zeta_{k-1}v_k-(\sigma_k v_k)\one\right].
\end{equation}

\begin{proof}[Proof of Lemma \ref{LmSmooth}]
We claim that functions $l,$ $\beta, \tbeta, \sigma, v$ and $\zeta$ are $C^\infty.$
The smoothness of $\zeta$ and $v$ is proven in \cite[Lemma 12.1]{DG4}, and the smoothness of
$\beta$ is proven in \cite[equation (12.2)]{DG4}. The smoothness of $\sigma$ and $l$ can be established similar to
$v$ and the smoothness of $\tbeta$ is similar to $\beta.$ \eqref{RhoDelta} now shows that $\brho$ is $C^\infty$ and moreover that
the first term in the formula for $\bDelta$ is $C^\infty$.
It remains to show that
$\tcB(\omega):=\cB(\omega)-\cB(\omega+\gamma)$ is $C^\infty.$
From \eqref{DefcB} it follows that
$$ \tcB(\omega)=\beta(\omega+\gamma) [v(\omega)-(\sigma(\omega) v(\omega)) \one]+
\sum_{k=0}^{\infty} \beta(\omega+(k+1)\gamma)\Lambda_k(\omega) $$
where
$$\Lambda_k(\omega)=[\zeta(\omega)-I] \zeta(\omega+\gamma) \dots\zeta(\omega+(k-1)\gamma)
v(\omega+k\gamma).
$$
In view of the foregoing discussion it remains to show that for each $r$ there exist constants $C_r>0$
and $\theta_r<1$ such that
\begin{equation}
\label{LambdaCrSmooth}
 ||\Lambda_k||_{C^r}\leq C_r \theta_r^k.
\end{equation}
Denote
$$ v_{k,l}(\omega)=\zeta(\omega+(k-1-l)\gamma) \dots \zeta(\omega+(k-1)\gamma) v(\omega+k\gamma), $$
$$w_{k,l}=\frac{v_{k,l}}{||v_{k,l}||}, \quad
\eta_{k,l}=\ln \frac{||v_{k,l}||}{||v_{k,l-1}||}.
$$
We have
$$\Lambda_k(\omega)=[\zeta(\omega)-I] v_{k, k-1}(\omega)
=[\zeta(\omega)-I] \exp\left(\sum_{l=0}^{k-1} \eta_{k,l}\right) w_{k, k-1}(\omega)$$
$$=[\zeta(\omega)-I] \exp\left(\sum_{l=0}^{k-1} \eta_{k,l}\right) \left[w_{k, k-1}(\omega)-\one\right] $$
where the last equality holds since $\zeta(\omega)$ is a stochastic matrix. Using this representation we can
deduce \eqref{LambdaCrSmooth} from the following inequalities.
\begin{equation}
\label{WCrSmooth}
 ||w_{k,l}-\one ||_{C^r}\leq C_r \theta_r^l,
\end{equation}
\begin{equation}
\label{EtaCrSmooth}
 ||\eta_{k,l}||_{C^r}\leq C_r \theta_r^l.
\end{equation}
Indeed \eqref{EtaCrSmooth} shows that
$\DS \left\Vert \sum_{l=0}^{k-1} \eta_{k,l}\right\Vert_{C^r}\leq \brC_r $
and so
$\DS \left\Vert \exp\left( \sum_{l=0}^{k-1} \eta_{k,l}\right) \right\Vert_{C^r}\leq \brrC_r .$

We note that, by the definition of $\eta_{k,l},$ \eqref{EtaCrSmooth} follows from
\eqref{WCrSmooth}, so it suffices to show the latter inequality. We shall use the following fact.
\begin{lemma}
\label{ContrSmooth}
Let $\Phi_j(x, u)$ be a family of contractions of a manifold $X$ depending on a parameter $u$
from an open set $D\subset\reals^d.$
That is, we assume that there exist constants $K>0$ and $\theta<1$ such that
\begin{equation}
\label{DefContr}
|| D_x \Phi_j ||\leq \theta
\end{equation}
and for some $r\geq 2$
$$ ||\Phi_j||_{C^r(X\times D)}\leq K. $$
Assume also that there exists a common fixed point for all values of the parameter, that is,
there exists $p\in X$ such that for all $u\in D$
\begin{equation}
\label{Common}
\Phi_j(p, u)\equiv p.
\end{equation}
Then there are constants $\brK>0, \brtheta<1$ such that
\begin{equation}
\label{CrEstContr}
 \left\Vert \Phi_l\circ \Phi_{l-1} \circ\dots \circ \Phi_1 \right\Vert_{C^r(X\times D)} \leq \brK \brtheta^l.
\end{equation}
\end{lemma}
To prove \eqref{WCrSmooth} we apply Lemma \ref{ContrSmooth} where $X$ is a neighborhood
of $\one$ in $(m-1)$-dimensional
projective space and $\Phi_l(w, \omega)=\zeta(\omega-(l-1)\gamma) w.$ To verify the conditions
of the lemma we note that $\zeta$ contracts the Hilbert metric on the positive cone and that
$\zeta(\cdot)\one\equiv \one$ since $\zeta$s are stochastic matrices. (See Appendix \ref{AppPM} for the
definition of the Hilbert's metric and the related contraction properties of positive matrices.)
This completes the proof of
Lemma~\ref{LmSmooth} modulo the proof of Lemma \ref{ContrSmooth} given below.
\end{proof}
\begin{proof}[Proof of Lemma \ref{ContrSmooth}]
Since the iterations of $\Phi$ converge to $p$ exponentially fast, we may assume that we
start in a small neighborhood of $p.$ By passing to local coordinates we may further assume
that $X$ is a bounded domain in $\reals^q$ for some $q.$

We prove \eqref{CrEstContr} by induction on $r.$ For $r=0$ the estimate follows by contraction mapping principle.
Let us now consider $r=1.$
Denoting
$$ x_l=(\Phi_l\circ \dots \circ \Phi_1)(x), \quad
A_l=D_x (\Phi_l\circ \dots \circ \Phi_1)x, \quad
B_l=D_u (\Phi_l\circ \dots \circ \Phi_1)x
$$
we get
\begin{equation}
\label{DerRec}
 A_l=D_x \Phi_l(x_{l-1}) A_{l-1}, \quad
B_l=D_x \Phi_l (x_{l-1}) B_{l-1}+D_u \Phi_l(x_l).
\end{equation}
Now the required bound for $A_l$ follows directly from \eqref{DefContr}.
To estimate $B_l$ we iterate the corresponding recurrence to get
$$ B_l=\sum_{j<l} D_x \Phi_{l} \dots D_x \Phi_{j+1} D_u \Phi_j(x_{j-1}) $$
To estimate the above sum we note that the terms with $j<l/2$ are exponetially small due to \eqref{DefContr}
while the terms with $j\geq l/2$ are exponentially small since \eqref{Common} implies that
$ D_u \Phi_j(p, u)\equiv 0 $ and so
$ D_u \Phi_j(x_{j-1})=O(\theta^j). $ This proves the claim for $r=1$ and completes the base of induction.

To perform the inductive step we assume that the Lemma holds for $r-1.$ In view of the foregoing discussion
to prove the result for $r$ we need to estimate $C^{r-1}$ norm of
$$(D_x (\Phi_l\circ\dots\circ \Phi_1),  D_u (\Phi_l\circ\dots\circ \Phi_1)).$$
In view of \eqref{DerRec} this reduces to studying the iterations of maps
$$ \hPhi_j(x, A, B, u)=(\Phi_j(x), (D_x \Phi_j) A, (D_x \Phi_j)B+D_u \Phi_j).$$
Since $\hPhi_j$ are contractions having common fixed point $(p, 0, 0)$ the required estimate is true by inductive assumption.
\end{proof}
\begin{proof}[Proof of Lemma \ref{LmHolder}]
We claim that functions $l,$ $\beta, \tbeta, \sigma, v$ and $\cB$ are H\"older continuous with respect to the metric
$\bd.$ In fact, the H\"older continuity of $\lambda$ and $v$ is proven in \cite[Appendix A]{DG4}.
The proof of H\"older continuity of $\tlambda, l, $ and $\sigma$ is very similar.
Next the H\"older continuity of $\beta$ and $\tbeta$ follows from the H\"older continuity of $\lambda$ and $\tlambda,$
relation \eqref{CoBoundary} and the Livsic Theorem \cite{Lv}. To prove the H\"older continuity of $\cB$ we note that
the second factor in the sum \eqref{DefcB} is exponentially small due to \eqref{prodZeta}.
Therefore the required statement is a consequence of Proposition \ref{PrHoldSum} below.
\end{proof}
\begin{proposition}
\label{PrHoldSum}
Given positive constants $a, c_1$ and $c_2$ there exists a constant \\
$b=b(a, c_1, c_2)$ such that if
$(X, \bd)$ is a metric space and
$$ H(x)=\sum_{k=1}^\infty H_k(x)$$
where
$$||H_k||_\infty\leq K e^{-c_1 k}, \quad ||H_k||_{C^a(X)}\leq K e^{c_2 k} $$
for some constant $K.$ Then $H\in C^b(X).$
\end{proposition}

\begin{proof}
For each $n$ we have the following estimate
$$ \left|H(x)-H(y)\right|\leq \left[\sum_{k=1}^{n-1} K \bd^a (x,y) e^{c_2 k}\right]+2
\sum_{k=n}^\infty K e^{-c_1 k}=
K \left[\frac{e^{c_2 n}-e^{c_2}}{e^{c_2}-1} \bd^a(x,y)+\frac{e^{-c_1 n}}{1-e^{-c_1}} \right]. $$
Choosing $n$ so that $e^{c_2 n} \bd^a(x,y)$ and $e^{-c_1 n}$ are of the same order we obtain the claim.
\end{proof}

The proof of Lemma \ref{LmPertMatr} relies on the following fact
\begin{proposition}
\label{PrContrNS}
Let $\Phi_n'(x)$ and $\Phi_n''(x)$ be two families of contractions of a bounded metric space
 $X.$ That is, assume that there are constants $K>0$ and $\theta<1$ such that
$ \diam(X)\leq K,$ and
for all $n\in\naturals$
$$ d(\Phi_n'(x_1), \Phi_n'(x_2))\leq \theta d(x_1, x_2),  \quad
d(\Phi_n''(x_1), \Phi_n'(x_2))\leq \theta d(x_1, x_2) . $$
If there are constants $C, \sigma$ such that for  each $x\in X$ and for all $n\in\naturals$
$$ d(\Phi_n'(x), \Phi_n''(x))\leq \frac{C}{n^\sigma} $$
then there is a constant $\brC$ such that for all $n\in\naturals$
$$ d_n:=d(\Phi'_n \dots \Phi'_1(x'), \Phi''_n \dots \Phi''_1(x''))\leq \frac{\brC}{n^\kappa}. $$
\end{proposition}

\begin{proof}
Iterating the estimate
$ d_n\leq \theta d_{n-1}+\dfrac{C}{n^\kappa} $ we obtain
$ \displaystyle d_n\leq \theta^{n-1} d_1+\sum_{j=1}^{n-1} \dfrac{C \theta^{n-j}}{j^\kappa}. $
Since $d_1\leq K$ the result follows.
\end{proof}

\begin{proof}[Proof of Lemma \ref{LmPertMatr}]
\eqref{AKey} follows from Proposition \ref{PrContrNS} since the map
relating $\zeta_n$ to $\zeta_{n-1}$ is a contraction in the total variation distance
(see \cite[Appendix D]{DG2}) while the maps relating $v_{n}$ to $v_{n-1}$ and $l_n$ to $l_{n-1}$ are contractions
in the Hilbert metric. \eqref{ACor} follows from \eqref{AKey} and the explicit formulas relating
$A_n,$ $\lambda_n$ and $\tlambda_n$ to $\zeta_n,$ $v_n$ and $l_n.$ Next
$$ \frac{\bar\beta_n}{\beta_n}=
\frac{\bar\beta_1}{\beta_1}
\prod_{j=1}^{n-1} \left(\frac{\brlambda_j}{\lambda_j}\right). $$
Since the above series converges due to \eqref{ACor} we obtain that
$\DS \beta_+=\lim_{n\to+\infty} \frac{\bar\beta_n}{\beta_n}$
exists. The existence of
$\DS \beta_-=\lim_{n\to-\infty} \frac{\bar\beta_n}{\beta_n}$
and
$\DS \tbeta_\pm=\lim_{n\to\pm\infty} \frac{\overline{\tbeta}_n}{\tbeta_n}$
are similar.

Next the existence of $\mu_\pm,$ $a_\pm$ and $b_\pm$ follows from the existence of the above limits in view
of the formulae
\begin{equation}
\label{FormRhoDelta}
 \rho_n=\frac{c l_n}{\tbeta_n},  \quad
\Delta_n=\beta_n \sigma_n v_n+\cB_{n+1}-\cB_n
\end{equation}
with
$$
 \cB_n=\sum_{k=n}^{\infty}\beta_{k+1} \left[\zeta_n\dots\zeta_{k-1}v_k-(\sigma_k v_k)\one\right].
$$
proven in \cite{DG4}.

It remains to show that
\begin{equation}
\label{SameLim}
\beta_+=\tbeta^+\quad\text{and}\quad \beta_-=\tbeta_-.
\end{equation}
In view of  \eqref{RhoMNorm}
$$\rho_n P_n (\fm_{n+1}-\zeta_{n+1}^- \fm_n)=
\brrho_n \brP_n (\bar{\fm}_{n+1}-\brzeta_{n+1}^- \bar{\fm}_n).$$
However due to \eqref{FormRhoDelta} the ratio of the RHS to the LHS for $n\to\pm\infty$
equals to $\frac{\beta_\pm}{\tbeta_\pm}(1+o_{n\to\pm\infty}(1))$
proving \eqref{SameLim}.
\end{proof}


\begin{thebibliography}{99}
\bibitem{Al} S. Alili:  {\it Asymptotic behaviour for random walks in
random  environments,} J. Appl. Prob. \textbf{36} (1999) 334--349.

\bibitem{BCR}
N. Berger, M. Cohen, R. Rosenthal: {\it Local limit theorem and equivalence of dynamic and static
points of view for certain ballistic random walks in i.i.d. environments,} Ann. Probab. {\bf 44} (2016) 2889--2979.

\bibitem{Bf} G. Birkhoff: {\it Lattice Theory}, AMS Colloquium publications, Volume 25.
Prvidence, Rhode Island, Third Edition (1973), vi+418p.

\bibitem{B} E. Bolthausen:
{\it Exact convergence rates in some martingale central limit theorems,}
Ann. Probab. {\bf 10} (1982) 672--688.

\bibitem{BG1} E. Bolthausen, I. Goldsheid: {\it Recurrence and transience of
random walks in random environments on a strip,} Comm. Math.
Phys. \textbf{214} (2000) 429--447.

\bibitem{BG2} E. Bolthausen, I. Goldsheid: {\it Lingering random walks in
random environment on a strip,} Comm. Math. Phys.  \textbf{278} (2008) 253--288.

\bibitem{BS} E. Bolthausen, A.-S. Sznitman: {\it Ten lectures on random media,} DMV Seminar, {\bf 32} (2002)
Birkhauser, Basel.

\bibitem{BS1} E. Bolthausen, A.-S. Sznitman: {\it On the static and dynamic poins of view for certain
random walks in random environment,} Methods and Applications of Analysis \textbf{9}  (2002), 345--376.

\bibitem{Bor} A. N. Borodin:
{\it The asymptotic behavior of local times of recurrent random walks with finite variance,}
Teor. Veroyatnost. i Primenen. 26 (1981) 769--783.

\bibitem{Br1} J. Br\'emont: {\it One-dimensional finite range random walk in random medium and invariant measure equation,}
Ann. Inst. Henri Poincar\'e Probab. Stat. 45 (2009) 70--103.

\bibitem{Br2} J. Br\'emont: Random walk in quasi-periodic random environment. Stochastics and Dynamics,
{\bf 9} (2009) 47--70.

\bibitem{Bn}  B. M. Brown: {\it Martingale central limit theorems,} Ann. Math. Statist. {\bf 42} (1971) 59--66.



\bibitem{CD} A. Chiarini, J.-D. Deuschel:
{\it Local central limit theorem for diffusions in a degenerate and unbounded random medium,}
Electron. J. Probab. {\bf 20} (2015), paper 112, 30 pp.

\bibitem{CFS} I. P. Cornfeld, S. V. Fomin, Ya. G. Sinai:
{\em Ergodic theory,}
Grundlehren der Mathematischen Wissenschaften
{\bf 245} (1982) Springer, New York, x+486 pp.

\bibitem{DMD}
B. Davis, D. McDonald: {\it An elementary proof of the local central limit theorem,}
J. Theoret. Probab. {\bf 8} (1995) 693--701.


\bibitem{DFS} D. Dolgopyat, B. Fayad, M. Saprykina:
{\it Erratic behavior for 1-dimensional random walks in generic quasi-periodic environment,}
in preparation.

\bibitem{DG2}
D. Dolgopyat, I.  Goldsheid: {\it Limit theorems for random walks on a strip in subdiffusive regimes,}
Nonlinearity {\bf 26} (2013) 1741--1782.

\bibitem{DG3}
D. Dolgopyat, I. Goldsheid: {\it Local Limit Theorems for random walks in a 1D random environment,}
Arch. Math. {\bf 101} (2013) 191--200.

\bibitem{DG4}
D. Dolgopyat, I.  Goldsheid: {\it Central Limit Theorem for recurrent random walks on a strip with bounded potential,}
Nonlinearity {\bf 31} (2018) 3381--3412.


\bibitem{DG5}
D. Dolgopyat, I.  Goldsheid:
{\it Invariant measure for random walks on ergodic environments on a strip,}
Ann. Prob. {\bf 47}, No. 4 (2019), 2494--2528.

\bibitem{Du}
R. Durrett: {\it Probability: theory and examples.} 4th edition.
Cambridge Univ. Press, Cambridge (2010).

\bibitem{GPS}
C. Gallesco, S. Popov, G. M. Sch\"utz:
{\it Localization for a random walk in slowly decreasing random potential,}
J. Stat. Phys. {\bf 150} (2013)  285--298.


\bibitem{G} V. P. Gaposhkin: {\it On the Dependence of the Convergence Rate in the Strong
Law of Large Numbers for Stationary Processes on the Rate of Decay of the
Correlation Function,} Th. Prob., Appl. {\bf 26} (1982) 706--720.


\bibitem{G2} I. Goldsheid: {\it Linear and Sub-linear Growth and the CLT
for Hitting Times of a Random Walk in Random Environment on a
Strip,} Prob. Th. Rel. Fields
 \textbf{141} (2008) 471--511.

\bibitem{HH} P. Hall, C. C. Heyde: {\it Martingale limit theory and its application,}
    Academic Press, New York-London, 1980. xii+308 pp.

\bibitem{HS} J. M. Harrison, L. A. Shepp:
{\it On skew Brownian motion,}
Ann. Probab. {\bf 9} (1981) 309--313.

\bibitem{JKO}  V. V. Jikov, S. M.  Kozlov, O. A. Oleinik:
{\it Homogenization of differential operators and integral functionals,}
Springer, Berlin, 1994. xii+570 pp.

\bibitem{KS1} V. Kaloshin, Ya. Sinai: {\it Nonsymmetric Random Walks along orbits of
Ergodic Automorphisms,} AMS. Transl. Ser. 2, Vol. {\bf 198} (2000) 8 pp.

\bibitem{K} H. Kesten:
{\it A renewal theorem for random walk in a random environment,}
Proc. Sympos. Pure Math., {\bf 31} (1977)  67--77.

\bibitem{KKS} H. Kesten, M. V. Kozlov, F. Spitzer:
{\it Limit law for random walk in a random environment.} Composito
Mathematica \textbf{30}, 145--168 (1975).

\bibitem{KLO} T. Komorowski, C.  Landim, S. Olla:
{\it Fluctuations in Markov processes.
Time symmetry and martingale approximation,}
Grundlehren der Mathematischen Wissenschaften {\bf 345} (2012) Springer, Heidelberg xviii+491 pp.

\bibitem{Koz} S. Kozlov: {\em The method of averaging and walks in inhomogeneous environments,}
Russian Math. Surveys {\bf 40} (1985) 73--145.

\bibitem{L} S. Lalley:
{\it An extension of Kesten's renewal theorem for random walk in a random environment,}
Adv. in Appl. Math. {\bf 7} (1986)  80--100.

\bibitem{Lam}
J. Lamperti: {\it A new class of probability limit theorems,}  J. Math. Mech. {\bf 11} (1962) 749--772.

\bibitem{Lej}
A. Lejay: {\it On the constructions of the skew Brownian motion,} Probab. Surv. {\bf 3} (2006), 413--466.

\bibitem{Len} M. Lenci: {\it On infinite-volume mixing,} Comm. Math. Phys. {\bf 298} (2010) 485--514.

\bibitem{LS} L. Leskela, M. Stenlund:
{\it A local limit theorem for a transient chaotic walk in a frozen environment,}
Stochastic Process. Appl. {\bf 121} (2011) 2818--2838.

\bibitem{Lv}
A. N. Livsic:
{\it Cohomology of dynamical systems,} (Russian) Izv. Akad. Nauk SSSR Ser. Mat. \textbf{36} (1972) 1296--1320.


\bibitem{MW1} M. Menshikov, A. Wade:
  {\it Random walk in random environment with asymptotically zero perturbation,}
  JEMS {\bf 8} (2006) 491--513.

\bibitem{MW2} M. Menshikov, A. Wade:
  {\it Logarithmic speeds for one-dimensional perturbed random walks in random environments,} Stoch. Process. Appl.
  {\bf 118} (2008) 389--416.

\bibitem{MW3} M. Menshikov, A. Wade:
{\it Rate of escape and central limit theorem for the supercritical Lamperti problem,}
Stochastic Process. Appl. {\bf 120} (2010) 2078--2099.


\bibitem{MPW} M. Menshikov, S. Popov, A. Wade:
{\it Non-homogeneous random walks: Lyapunov function methods for near-critical stochastic systems,}
 Cambridge Tracts in Math. {\bf 209} (2017) xviii+363 pp, Cambridge Univ. Press, Cambridge, UK.


\bibitem{Pr} Yu. V. Prokhorov: {\it On the local limit theorem for lattice distributions,} Dokl. Akad. Nauk SSSR
{\bf 98} (1954) 535--538.

\bibitem{R} A. Roitershtein: {\it Tranzient random walks on a strip in a random environment,} Ann. Probab.
{\bf 36} (2009) 2354--2387.

\bibitem{S1} Ya. G. Sinai: {\it The limiting behavior of a one-dimensional
random walk in a random medium,} Theory Prob. Appl. \textbf{27} (1982)
256--268.

\bibitem{S2} Ya. G. Sinai: {\it Simple random walks on tori,} J. Statist. Phys. {\bf 94} (1999) 695--708.

\bibitem{So} F. Solomon: {\it Random walks in a random environment,} Ann.
Prob. \textbf{3} (1975) 1--31.

\bibitem{Stn} M. Stenlund {\it A local limit theorem for random walks in balanced environments,}
Electron. Commun. Probab. {\bf 18} (2013) paper 19, 13 pp.

\bibitem{St} C. Stone: {\it Limit theorems for random walks, birth and death processes, and diffusion processes,}
Illinois J. Math. {\bf 7} (1963) 638--660.

\bibitem{Z} O. Zeitouni: Random walks in random environment, XXXI
Summer school in Probability, St. Flour (2001). Lecture notes in
Math. {\bf 1837}, 193--312, Springer, Berlin, 2004.
\end{thebibliography}
\end{document}